\documentclass{amsart}

\usepackage[all,hyperref]{bi-discrete}

\usepackage[backend=biber,maxbibnames=6,maxalphanames=5,style=alphabetic,bibencoding=utf8,giveninits,url=false,isbn=false, maxcitenames=5,
mincitenames=5]{biblatex}

\usepackage{tikz}
\usetikzlibrary{intersections,calc,arrows.meta,patterns,angles,quotes}
\usetikzlibrary{positioning}
\newcommand{\btimes}{\mathbin{\rotatebox[origin=c]{90}{$\Join$}}}

\usepackage[backend=biber,maxbibnames=6,maxalphanames=5,style=alphabetic,bibencoding=utf8,giveninits,url=false,isbn=false, maxcitenames=5,
mincitenames=5]{biblatex}
\usepackage{tikz}
 \providecommand{\frS}{{\mathfrak S}}
\usetikzlibrary{intersections,calc,arrows.meta,patterns,angles,quotes}
\usetikzlibrary{positioning}

\providecommand{\phidx}{\ensuremath{{\phi^*(\cdot)}}}

\begin{document}

\author[A.  Balci]{Anna  Balci}
\address[A. Kh. Balci]{Charles University  in Prague, Department of Mathematical Analysis
Sokolovsk\'a 49/83, 186 75 Praha 8, Czech Republic and University of Bielefeld, Department of Mathematics,  Postfach 10 01 31, 33501 Bielefeld, Germany.}
\email{akhripun@math.uni-bielefeld.de}
\urladdr{http://www.bi-discrete.com}
\thanks{The research of A. Kh. Balci and L. Diening is funded by the Deutsche Forschungsgemeinschaft (DFG, German Research Foundation) --- SFB 1283/2 2021 --- 317210226.}
\thanks{The Research of Anna Kh. Balci was also  supported by Charles University  PRIMUS/24/SCI/020 and Research Centre program No. UNCE/24/SCI/005}

\author[L. Diening]{Lars Diening}
\address[L. Diening]{Department of Mathematics, University of Bielefeld, Postfach 10 01 31, 33501 Bielefeld, Germany.}
\email{lars.diening@math.uni-bielefeld.de}
\urladdr{http://www.bi-discrete.com}

\author[M. Surnachev]{Mikhail Surnachev}
\address[M.Surnachev ]{Keldysh Institute of Applied Mathematics, Miusskaya sq., 4 Moscow, 125047, Russia}
\email{peitsche@yandex.ru}
\thanks{Research of Mikhail Surnachev is supported by Moscow Center of Fundamental and Applied Mathematics, Agreement with the Ministry of Science and Higher Education of the Russian Federation, No. 075-15-2022-283.}

\title{Scalar minimizers with maximal singular sets and lack of Meyers property}



\begin{abstract}
  We present a general procedure to construct examples of convex scalar variational problems which admit a minimizers with large singular sets. The dimension of the set of singularities is maximal and the minimizer has no higher integrability property (failure of Meyers property).
\end{abstract}

\maketitle
{\centering\footnotesize Dedicated to Nina Nikolaevna Uraltseva \par}


\section{Introduction}

We construct new examples of irregular minimizers for nonautomous integrands with non-standard growth conditions
$$
\mathcal{F}(u):=\int\limits_{\Omega} \phi(x,\abs{\nabla u})\, dx \to \min.
$$
Here $u$ takes prescribed (smooth) boundary value on $\partial \Omega$, where $\Omega$ is a bounded Lipschitz domain in $\mathbb{R}^d$. Our examples of potentials with non-standard growth include the double phase energy $\phi(x,t) = t^p+a(x)t^q$  and the $p(x)$-energy $\phi(x,t)=t^{p(x)}$.  Fonseca, Mal\'y and Mingione studied in \cite{FonMalMin04} the size of possible singular sets for minimizers of the double phase potential. They  constructed for every~$\varepsilon>0$ a minimizer for a double phase energy with $a\in C^\alpha(\Omega)$ and 
\begin{align}\label{dimensional_threshold}
1 < p < d < d+\alpha < q < \infty,
\end{align}
which has the Hausdorff dimension of the set of non-Lebesgue points larger than $d-p-\varepsilon$. This result is almost sharp in the sense that the Hausdorff dimension of the singular set of any $W^{1,p}(\Omega)$ function is at most $d-p$.

In this paper we improve these results and extend them also to other models. 
%
First, we construct a double phase energy and a minimizer with the optimal Hausdorff dimension of the singular set~$\Sigma$ in the sense that 
\begin{align*}
 \dim_{\mathcal{H}}(\Sigma)=d-p.
\end{align*}
This is the largest possible value for any $W^{1,p}(\Omega)$ function. In such a way we get rid of the $\varepsilon>0$ in~\cite{FonMalMin04}.  Second, we get rid of the dimensional threshold~$p<d<q$ by using the fractal Cantor set barriers from~\cite{BalDieSur20lavrentiev}. Our results hold in the range
\begin{align*}\
q > p+\alpha \max \left\{1, \frac{p-1}{d-1}\right\}.
\end{align*}
These bounds are sharp in view of the regularity results of \cite{ColMin15,BarColMin18} as explained in~\cite[Remark~35]{BalDieSur20lavrentiev}. For $1<p\leq d$ we obtain a minimizer with a large singular set, and for all $p>1$ we obtain minimizers without Meyers property. This means that~$\nabla v\not\in L^s(\Omega)$ for any~$s>p$. 

 
In~\cite{Zhi86}, \cite{Zhi95} Zhikov for the two-dimensional case and $p<2<q$ constructed an example exhibiting the Lavrentiev phenomenon, i.e. 
\begin{equation}\label{Lavr}
 \inf  \mathcal{F}(W^{1,\phi(\cdot)}(\Omega)) < \inf  \mathcal{F}(C_0^\infty(\Omega)).
\end{equation}


This example, known in the literature as one-saddle point example, allows also to  to construct examples of singular minimizers of functionals with coefficients having a cone-type singularity. Later this construction was used in~\cite{EspLeoMin04} to design one-saddle point examples for any dimension under the condition~\eqref{dimensional_threshold}. In~\cite{FonMalMin04} Fonseca, Mingione and Mal{\'y}  used this set-up and  an iterative process to  build local minimizers having fractal singular sets under the same condition~\eqref{dimensional_threshold}. The restriction $p<d<q$ is known as \textit{the dimensional threshold}.

Let us mention  higher regularity results for double phase model. In fact, if $\frac{q}{p} \leq 1+\frac \alpha{d}$
and $a \in C^{0,\alpha}$, then minimizers of~$\mathcal{F}$ are
automatically in~$W^{1,q}$, see \cite{ColMin15}. Moreover, bounded
minimizers of~$\mathcal{F}$ are automatically~$W^{1,q}$
if~$a \in C^{0,\alpha}$ and~$q \leq p+\alpha$,
see~\cite{BarColMin18}. If the minimizer is from~$C^{0,\gamma}$, then
the requirement can be relaxed to $q \leq p + \frac{\alpha}{1-\gamma}$, see 
\cite[Theorem~1.4]{BarColMin18}.
 
In~\cite{BalDieSur20lavrentiev} the authors of the present paper presented new examples for the Lavrentiev gap and non-density of smooth functions without the dimensional threshold restriction by using fractals. 
We applied our method to the setting of variable exponents~$\phi(x,t)=t^{p(x)}$, the double phase potential and weighted integrand~$\phi(x,t)=a(x)t^2$. Later Balci and Surnachev in~\cite{BalSur21} applied the same construction to the model
\begin{align*}
\phi(x,t)=t^p\log^{-\beta}(e+t)+a(x)t^p\log^\alpha(e+t).
\end{align*}
The regularity properties of the integrand of this type for~$\Phi(x,t)$ with~$\beta=0$ were studied firstly by Baroni,  Colombo,  Mingione in~\cite{BarColMin15} where it  was called~\textit{the borderline case of double phase potential}. In particular, they obtained the ~$C^{0,\gamma}_{\loc}$ regularity result for the minimizers provided that the weight~$a(x)$ is~$\log$-H\"older continuous (with some~$\gamma$) and a more strong result (any~$\gamma\in (0,1)$) for the case of vanishing~$\log$-Hölder continuous  weight.

The variable exponent model~$\phi(x,t)=t^{p(x)}$ was studied  in numerous papers. We refer to the monographs~\cite{Zhi11,DieHHR11,CruFio13,KokMesRafSam16}. The results (with a few exceptions) mainly concern the case of log-H\"older exponents. For the variable exponent model the Hölder regularity of solutions, a Harnack type  inequality for non-negative solutions, and boundary regularity results were obtained in \cite{Alk97,Alk05,AlkKra04}. Gradient regularity for Hölder exponent was obtained by  \cite{CosMin99} and for the log-Hölder exponents by \cite{AceMin01} and Ok~\cite{Ok17}. These papers as well as~\cite{Zhi97Meyers} also contain Meyer type properties. For a class of exponents with the modulus of continuity slightly worse than  $\log$-H\"older, a logarithmic Meyers property was obtained in~\cite{ZhiPas08}.

Let us also mention, that computing such type of minimizers  numerically is a challenging problem, since the
standard numerical schemes fail to converge to the~$W$-minimiser of the problem. For the scalar case the problem could be solved using non-conforming methods, see Balci, Ortner, and Storn~\cite{BalOrtSto22}.
 
\subsection{Models}

We consider integral functionals of the form 
$$
\mathcal{F}(u) = \int\limits_\Omega \varphi(x,|\nabla u|)\, dx
$$
and treat the following models.
\begin{itemize}[nosep]
\item  The double phase energy: for $1<p<q<\infty$ and a nonnegative weight $a=a(x)$ we define 
\begin{equation}\label{dblphasedef}
\begin{gathered}
\varphi(x,t) = \frac{t^p}{p} + a(x) \frac{t^q}{q},\\
\mathcal{F}(u) = \int\limits_\Omega \bigg(\frac{|\nabla u|^{p}}{p} + a(x) \frac{|\nabla u|^q}{q}\bigg)\, dx.
\end{gathered}
\end{equation}
\item The borderline double phase: for $1<p<\infty$, a nonnegative weight $a=a(x)$, and $\alpha,\beta \in \mathbb{R}$ we define 
\begin{equation}\label{dbl2phasedef}
\begin{gathered}
\varphi(x,t) = t^p \log^{-\beta}(e+t) + a(x) t^p \log^\alpha(e+t),\\
\mathcal{F}(u) = \int\limits_\Omega |\nabla u|^p\bigl(\log^{-\beta}(e+|\nabla u|) + a(x) \log^\alpha (e+|\nabla u|) \bigr)\, dx.
\end{gathered}
\end{equation}
 \item  The $p(\cdot)$-energy: for $1<p_{-}<p_{+}<\infty$ let $p:\Omega \to [p_{-},p{+}]$ be a measurable function. Then we define 
\begin{equation}\label{varexpdef}
\varphi(x,t) = \frac{t^{p(x)}}{p(x)}, \quad \mathcal{F}(u) = \int\limits_\Omega \frac{|\nabla u|^{p(x)}}{p(x)}\, dx.
\end{equation}
\end{itemize}
 
 \subsection{Main Results}
 
Let us state our main results. Let $\Omega = (-1,1)^d$.

\begin{theorem}[double phase] \label{Theorem1}
For every 
$$
p>1,\quad \alpha\geq 0\quad \text{and}\quad q> p+\alpha \max \biggl\{1, \frac{p-1}{d-1}\biggr\}
$$ 
there exist a functional~$\mathcal{F}$ of the form \eqref{dblphasedef} with the weight~$a\in C^\alpha(\overline\Omega)$,~$a\ge 0$, and a local minimizer~$u\in W^{1,p}(\Omega)$ of the functional~$\mathcal{F}$ such that
$u\notin W^{1,s}$ for any~$s>p$.

If $p<d$, then there exists a closed set~$\Sigma\subset\Omega$ of non-Lebesgue points of $u$ with
$ \dim_{\mathcal{H}}(\Sigma)=d-p$.

\end{theorem}

Theorem~\ref{Theorem1} follows from Theorems~\ref{T:dbl_sub}, \ref{T:dbl_super} in the sub- and super dimensional cases, respectively.

\begin{theorem}[borderline double phase]\label{Theorem2}
For every choice of the parameters
\begin{equation}\label{dbl2ass}
p>1,\quad \alpha,\beta \in \mathbb{R},\quad \varkappa \geq 0,\quad \alpha+\beta>p+\varkappa
\end{equation}
there exist a functional~$\mathcal{F}$ of the form \eqref{dbl2phasedef} where~$a(\cdot)$ has a modulus of continuity $C\log^{-\varkappa} (e+\rho^{-1})$, and a local minimizer~$u\in W^{1,p}(\Omega)$ of the functional~$\mathcal{F}$ such that $u\notin W^{1,s}$ for any~$s>p$.

If $1<p<d$, then there exists a closed set~$\Sigma\subset\Omega$ of non-Lebesgue points of $u$  with $ \dim_{\mathcal{H}}(\Sigma)=d-p$.
\end{theorem}

Theorem~\ref{Theorem2} follows from Theorems~\ref{T:bdd_sub}, \ref{T:bdd_super} in the sub- and super dimensional cases, respectively.

\begin{theorem}[piecewise constant variable exponent]\label{Theorem3}
For any $1<p_{-}<p_{+}<\infty$ there exists a variable exponent $p:\Omega\to \{p_{-},p_{+}\}$ such that for the functional $\mathcal{F}$ defined by \eqref{varexpdef} there exists a local minimizer $u$ with $u\notin W^{1,s}$ for any~$s>p_{-}$. 

If $p_{-}< d$, then there exists a closed set~$\Sigma\subset\Omega$ of non-Lebesgue points of $u$ with $\dim_{\mathcal{H}}(\Sigma)=d-p_{-}$.
\end{theorem}

Theorem~\ref{Theorem1} follows from Theorems~\ref{T:varexp_sub}, \ref{T:varexp_super} in the sub- and super dimensional cases, respectively.

\begin{theorem}[continuous variable exponent]\label{Theorem4}
Let $p_0>1$ and $\varkappa>p_0/2$. There exists a variable exponent $p\in C(\overline\Omega)\cap C^\infty(\overline\Omega \setminus \frS)$, where $\frS$ is a closed set in $(-1,1)^{d-1}\times\{0\}$ of zero Lebesgue measure, with $p=p_0$ on $\frS$ and modulus of continuity
\begin{equation}\label{sigma_def1}
\sigma(t)=\kappa \frac{\log\log(e^3+t^{-1})}{\log(e+t^{-1})},
\end{equation}
such that for the functional~$\mathcal{F}(u)$ of the form \eqref{varexpdef} there exists a local minimizer~$u\in W^{1,p(\cdot)}(\Omega)$ such that $u\notin W^{1,p(\cdot)+\varepsilon}$ for any~$\varepsilon>0$.

If $p_0<d$, then there exists a closed set~$\Sigma\subset\Omega$ of non-Lebesgue points of $u$  with $ \dim_{\mathcal{H}}(\Sigma)=d-p_0$.
\end{theorem} 
  
Theorem~\ref{Theorem1} follows from Theorems~\ref{T:varexp2_sub}, \ref{T:varexp2_super} in the sub- and super dimensional cases, respectively.
 
In all of these theorems for $p=d$ ( in Theorems~\ref{Theorem1}, \ref{Theorem2}), $p_{-}=d$ (in Theorem~\ref{Theorem3}), $p_0=d$ (in Theorem~\ref{Theorem4}) the singular set of the minimizer will be a nonempty closed set of cardinality of the continuum of zero Hausdorff dimension (for \eqref{dbl2phasedef} under the additional assumption $\beta>1$).

\subsection{Outline} In Section~\ref{sec:usef-defin-lemm} we  introduce the models, related energy spaces and some auxiliary results. The general conditional results for the irregularity of the minimizers in the  subcritical case~$p<d$ are presented in Section~\ref{sec:sub}. We apply those results to the particular models in  Section~\ref{sec:sub_mod}. The general conditional results for the irregularity of the minimizers in the  supercritical  case~$p>d$ are presented in Section~\ref{sec:super}. We apply those results to the particular models in Section~\ref{sec:super_mod}. 

\section{Definitions and auxiliary results}\label{sec:usef-defin-lemm}

In this section we describe the functional setup for our variational problems and recall the construction of generalized Cantor sets. 

\subsection{Orlicz functions and functional spaces}

We say that~$\phi\,:\, [0,\infty) \to [0,\infty]$ is an Orlicz
function if $\phi$ is convex, left-continuous, $\phi(0)=0$,
$\lim\limits_{t \to 0} \phi(t)=0$ and
$\lim\limits_{t \to \infty} \phi(t)=\infty$.  The conjugate Orlicz
function~$\phi^*$ is defined by
\begin{align*}
  \phi^*(s) &:= \sup_{t \geq 0} \big( st - \phi(t)\big).
\end{align*}
In particular, $st \leq \phi(t) + \phi^*(s)$.

In the following we assume that
$\phi\,:\, \Omega \times [0,\infty) \to [0,\infty]$ is a generalized
Orlicz function, i.e. $\phi(x, \cdot)$ is an Orlicz function for
every~$x \in \Omega$ and $\phi(\cdot,t)$ is measurable for
every~$t\geq 0$. We define the conjugate function~$\phi^*$ pointwise,
i.e.  $\phi^*(x,\cdot) := (\phi(x,\cdot))^*$.

We further assume the following additional properties:
\begin{itemize}
\item We assume that~$\phi$ satisfies the $\Delta_2$-condition,
  i.e. there exists~$c \geq 2$ such that for
  all~$x \in \Omega$ and all~$t\geq 0$
  \begin{align}
    \label{eq:phi-Delta2}
    \phi(x,2t) &\leq c\, \phi(x,t). 
  \end{align}
\item We assume that~$\phi$ satisfies the~$\nabla_2$-condition,
  i.e. $\phi^*$ satisfies the~$\Delta_2$-condition. As a consequence,
  there exist~$s>1$ and~$c> 0$ such that for all $x\in \Omega$,
  $t\geq 0$ and $\gamma \in [0,1]$ there holds
  \begin{align}
    \label{eq:phi-Nabla2}
    \phi(x,\gamma t) \leq c\,\gamma^s \,\phi(x,t).
  \end{align}
\item We assume that $\phi$ and~$\phi^*$ are proper, i.e. for
  every~$t\geq 0$ there holds $\int\limits_\Omega \phi(x,t)\,dx< \infty$ and
  $\int\limits_\Omega \phi^*(x,t)\,dx < \infty$.
\end{itemize}
Let $L^0(\Omega)$ denote the set of measurable function on~$\Omega$
and $L^1_{\loc}(\Omega)$ denote the space of locally integrable 
functions. 
We define the generalized Orlicz norm by
\begin{align*}
  \norm{f}_{{\phi(\cdot)}} &:= \inf \biggset{\gamma > 0\,:\, \int\limits_\Omega
                     \phi(x,\abs{f(x)/\gamma})\,dx \leq 1}.
\end{align*}

Then generalized Orlicz space~$L^{{{{\phi(\cdot)}}}}(\Omega)$ is defined as the
set of all measurable functions with finite generalized Orlicz norm
\begin{align*}
  L^{{{{\phi(\cdot)}}}}(\Omega) &:= \bigset{f \in L^0(\Omega)\,:\, \norm{f}_{{{\phi(\cdot)}}}<\infty
                      }.
\end{align*}

For example the generalized Orlicz function $\phi(x,t) = t^p$
generates the usual Lebesgue space~$L^p(\Omega)$.

The $\Delta_2$-condition of~$\phi$ and~$\phi^*$ ensures that our space
is uniformly convex. The condition that~$\phi$ and $\phi^*$ are proper
ensure that $L^{{{{\phi(\cdot)}}}}(\Omega) \embedding L^1(\Omega)$ and
$L^{\phidx}(\Omega) \embedding L^1(\Omega)$. Thus $L^{{{{\phi(\cdot)}}}}(\Omega)$
and $L^{\phidx}(\Omega)$ are Banach spaces.

We define the generalized Orlicz-Sobolev space~$W^{1,{{{\phi(\cdot)}}}}$ as
\begin{align*}
  W^{1,\phi(\cdot)}(\Omega) &:= \set{w \in W^{1,1}(\Omega)\,:\,
                        \nabla w \in L^{\phi(\cdot)}(\Omega)},
\end{align*}
with the norm
\begin{align*}
  \norm{w}_{1, \phi(\cdot)}:=\norm{w}_{1}+\norm{\nabla w}_{ \phi(\cdot)}.
\end{align*}

In general smooth functions are not dense
in~$W^{1,\phi(\cdot)}(\Omega)$. Therefore, we define~$H^{1,\phi(\cdot)}(\Omega)$
as
\begin{align*}
  H^{1,\phi(\cdot)}(\Omega) &:= \big(\text{closure of}\  C^\infty(\Omega) \cap W^{1,\phi(\cdot)}(\Omega)\  \text{in}\  W^{1,\phi(\cdot)}(\Omega)\big).
\end{align*}
See \cite{DieHHR11} and~\cite{HarHas19} for further properties of
these spaces.

We also introduce the corresponding spaces with zero boundary values
as
\begin{align*}
  W^{1,{{{\phi(\cdot)}}}}_0(\Omega) &:= \set{w \in W^{1,1}_0(\Omega)\,:\,
                          \nabla w \in L^{\phi(\cdot)}(\Omega)}
\end{align*}
with same norm as in~$W^{1,\phi(\cdot)}(\Omega)$. And the corresponding space of smooth functions is defined as  
\begin{align*}
  H_0^{1,\phi(\cdot)}(\Omega) &:= \big(\text{closure of}\  C_0^\infty(\Omega) \cap W^{1,\phi(\cdot)}(\Omega)\  \text{in}\  W^{1,\phi(\cdot)}(\Omega)\big). 
\end{align*}
The space~$W_0^{1,\phi(\cdot)}(\Omega)$ are exactly those function, which can
be extended by zero to~$W^{1,\phi(\cdot)}(\setR^d)$ functions.

\subsection{Variational problems and the Lavrentiev phenomenon}

Let us define the energy~$\mathcal{F}\,:\, W^{1,{{{\phi(\cdot)}}}}(\Omega) \to \setR$ by
\begin{align*}
  \mathcal{F}(w) &:= \int\limits_\Omega \phi(x,\abs{\nabla w(x)})\,dx.
\end{align*}
In the language of function spaces~$\mathcal{F}$ is a semi-modular
on~$W^{1,{{{\phi(\cdot)}}}}(\Omega)$ and a modular on~$W^{1,{{{\phi(\cdot)}}}}_0(\Omega)$.
Let us start by introducing spaces with boundary values: for
$g \in H^{1,{{{\phi(\cdot)}}}}(\Omega)$ we define
\begin{align*}
  H_g^{1,{{{\phi(\cdot)}}}}(\Omega) &:= g + H_0^{1,{{{\phi(\cdot)}}}}(\Omega).
\end{align*}
Let 
\begin{align*}
  h_H(g) &= \argmin \mathcal{F}\big(H^{1,{{{\phi(\cdot)}}}}_g(\Omega)\big).
\end{align*}
Then
\begin{align*}
  -\Delta_{\phi(\cdot)} h_H(g) :=-\divergence \bigg( \frac{\phi'(x,\abs{\nabla h_H(g)})}{\abs{\nabla
  h_H(g)}} \nabla h_H(g) \bigg) &= 0 \qquad \text{ in  } (H^{1,\phi(\cdot)}_0(\Omega))^*.
\end{align*}

For $g \in W^{1,\phi(\cdot)}(\Omega)$ we define
\begin{align*}
  W_g^{1,\phi(\cdot)}(\Omega) &:= g + W_0^{1,\phi(\cdot)}(\Omega).
\end{align*}
Let
$$
  h_W(g) = \argmin \mathcal{F}\big(W^{1,\phi(\cdot)}_{g}(\Omega)\big).
$$
Formally, it satisfies the Euler-Lagrange equation (in the weak sense)
\begin{align*}
  -\Delta_{{{\phi(\cdot)}}} h_W(g) := -\divergence \bigg( \frac{\phi'(x,\abs{\nabla h_W(g)})}{\abs{\nabla
  h_W(g)}} \nabla h_W(g) \bigg) &= 0 \qquad \text{ in  }  (W^{1,\phi(\cdot)}_0(\Omega))^*,
\end{align*}
where $\phi'(x,t)$ is the derivative with respect to~$t$.

If $\phi(x,t)=\frac 12 t^2$, then $\Delta_{\phi(\cdot)}$ is just the standard Laplacian.
If $\phi(x,t)=\frac 1p t^p$, then $\Delta_{\phi(\cdot)}$ is the $p$-Laplacian.

The Lavrentiev phenomenon is the fact that for $g\in H^{1,\phi(\cdot)}(\Omega)$
$$
\min \mathcal{F}\big(W^{1,\phi(\cdot)}_{g}(\Omega)\big)< \min \mathcal{F}\big(H^{1,\phi(\cdot)}_{g}(\Omega)\big),
$$
and in particular $h_H(g)\neq h_W(g)$.

In our previous work \cite{BalDieSur20lavrentiev} we investigated in detail the Lavrentiev phenomenon for several models (variable exponent, double phase model, weighted $p$-energy). In this work we show that under similar assumptions the irregularity of the $W$-minimizer.

\subsection{Generalized Cantor sets and their properties.}\label{sec:Cantor}\hskip 0pt


Let $\{l_j,\ j=0,1,2,\ldots\}$ be a decreasing sequence of positive numbers starting from $l_0=1$:
$$
1=l_0>l_1>l_2>\ldots
$$
such that $l_{j-1} >2l_j$ for all $j\in \mathbb{N}$. We start from $I_{0,1}=[-1/2,1/2]$. On each $m$-th step we remove the open middle part of length $l_j-2l_{j+1}$ from the interval $I_{m,j}$, $j=1,\ldots, 2^m$ to obtain the next generation set of closed intervals $I_{m+1,j}$, $j=1,\ldots, 2^{m+1}$. The union of the closed intervals  $I_{m,j}  = [a_{m,j}, b_{m,j}]$, $j=1,\ldots, 2^m$ of length $l_m$  from the same generation forms the pre-Cantor set $C_m = \bigcup_{j=1}^{2^m} I_{m,j}$. The Cantor set $\frC = \bigcap_{m=0}^\infty C_m$ is the intersection of all pre-Cantor sets $C_m$.  

On each $m$-th step  we define the pre-Cantor measure as $\mu_m = |C_m|^{-1}\mathbbm{1}_{C_m}$, where $|C_m|= 2^m l_m$ is the standard Lebesgue measure of $C_m$, and the weak limit of the measures $\mu_m$ is the Cantor measure $\mu_{\frC}$ corresponding to $\frC$.

We require further that
$$
l_{m-1} - 2l_m > l_m - 2 l_{m+1} \Leftrightarrow l_{m+1} > \frac{3l_m - l_{m-1}}{2}, 
$$
at least for all sufficiently large $m$.

If the sequence $l_j$ satisfies the conditions above only for sufficiently large $j\geq j_0$,  then we modify it by taking the sequence $\tilde l_j = l_{j+j_0} (l_{j_0})^{-1}$, $j=0,1,2,\ldots$

For $k\in \mathbb{N}$ by $\frC^k$ and $\mu^{k}$ we denote the Cartesian powers of $k$ copies of $\frC$ and its corresponding Cantor measure, respectively.


\begin{definition}[Generalized Cantor sets]

\begin{itemize} \hphantom{Let it be}

\item ({\it Generalized $(\lambda,\gamma)$-Cantor sets}). Let $l_j = \lambda^j j^\gamma$, $\lambda\in (0,1/2)$, $\gamma\in\mathbb{R}$. We denote the corresponding Cantor set by $\frC_{\lambda,\gamma}$, the Cartesian product of its $k$ copies is $\frC^k_{\lambda,\gamma}$. For $\frC^k_{\lambda,\gamma}$ we denote $\frD = -k\log 2 / \log \lambda$, so that $\lambda^{\frD} = 2^{-k}$. We denote the Cantor measure corresponding to $\frC_{\lambda,\gamma}$ by $\mu_{\lambda,\gamma}$ and its $k$-th Cartesian power by $\mu^k_{\lambda,\gamma}$.

\item ({\it Meager Cantor sets}).  Let $l_j = \exp(-2^{j/\gamma})$, $\gamma>0$. Denote  the corresponding Cantor set by $\frC_{0,\gamma}$, and its Cartesian products by $\frC^k_{0,\gamma}$. For these sets we denote $\frD=0$. We denote the corresponding Cantor measures by $\mu^k_{0,\gamma}$.
\end{itemize}
\end{definition}

The standard $1/3$-Cantor set corresponds to the choice $l_j=(1/3)^j$, in our notation it is $\frC_{1/3,0}$, its fractal dimension is $\frD =\log_3 2$ and the corresponding measure is $\mu_{1/3,0}$. 


Denote $\frC_t = \{\mathrm{dist}(\bar x, \frC)<t\}$, where $\frC$ is one of the Cantor sets $\frC_{\lambda,\gamma}^k$ or~$\frC^k_{0,\gamma}$ defined above. Denote  by  $d_\infty (\bar x,\frC)$ the distance from $\bar x$ to $\frC$ in the maximum norm and let $\frC_{*,t} = \{d_\infty(\bar x, \frC)<t\}$. It is clear that $\frC_t \subset \frC_{*,t}$. In the following lemma $B^{m}_t(\bar x)$ is the open ball in  $\mathbb{R}^m$ with center at $\bar x$ and radius $t$. 

\begin{lemma}
  \label{lem:cantor-estimates}
  Let~$\lambda \in [0,\frac 12)$, $\gamma\in\mathbb{R}$, $1\leq m \leq d$. Let 
  $\frD = -m \log(2)/\log(\lambda)$ if $\lambda>0$. We   use the notation $x=(\bar{x},\hat{x}) \in \setR^m \times \setR^{d-m}$.
  Then we have the following properties: 
  \begin{itemize}
    \setlength{\itemsep}{1ex}%
  \item  For  every
    ball~$B^m_r(\bar{x})$, $r<1$, there
    holds
    \begin{align}\label{itm:cantor-estimates1}
     \begin{split}
    \mu^m_{\lambda,\gamma} (B^m_r(\bar{x})) \lesssim
    \indicator_{\set{d(\bar{x},\frC^m_{\lambda,\gamma})\le r}} r^{\frD} \log^{-\gamma \frD}(e+ r^{-1})\quad  \text{if}\quad \lambda>0,\\
    \mu^m_{0,\gamma} (B^m_r(\bar{x})) \lesssim
    \indicator_{\set{d(\bar{x},\frC^m_{0,\gamma})\le r}} \log^{-\gamma m}(e+ r^{-1}).
    \end{split}
\end{align}
   \item  For all $r\in(0,1)$ there holds
     \begin{align}\label{itm:cantor-estimates2}
     \begin{split}
    \mathcal{L}^m\big( \set{ \bar{x}: d(\bar{x},\frC^m_{\lambda,\gamma}) \leq  r } \big) \lesssim r^{m-\frD} \log^{\gamma \frD}(e+ r^{-1})\quad  \text{if}\quad \lambda>0,\\
     \mathcal{L}^m\big( \set{ \bar{x}: d(\bar{x},\frC^m_{0,\gamma}) \leq  r } \big) \lesssim r^{m} \log^{\gamma m}(e+ r^{-1}).
        \end{split}
     \end{align}
 \item 
 For all $\tau \in (0,4]$, $s \geq 0$, there holds 
     \begin{align}\label{itm:cantor-estimates4}
     \begin{split}
       &\bigabs{\big((\mu_{\lambda,\gamma}^m \times \delta_0^{d-m}) * (
       \indicator_{\set{\abs{\bar{y}}\leq \tau\abs{\hat{y}}}} \abs{y}^{-s})\big)(x)} \\
       &\lesssim
       \indicator_{ \set{ 
       d(\bar{x},\frC^m_{\lambda,\gamma}) \leq\tau \abs{\hat{x}}}}
       \,\abs{\hat{x}}^{\frD-s} \log^{-\gamma\frD}(e+ |\hat x|^{-1}),\quad \text{if}\quad \lambda>0,\\
       &\bigabs{\big((\mu_{0,\gamma}^m \times \delta_0^{d-m}) * (
      \indicator_{\set{\abs{\bar{y}}\leq \tau\abs{\hat{y}}}} \abs{y}^{-s})\big)(x)} 
      \lesssim\\
      &\indicator_{ \set{ 
      d(\bar{x},\frC^m_{0,\gamma}) \leq\tau \abs{\hat{x}}}}
      \,\abs{\hat{x}}^{-s} \log^{-\gamma m} (e+|\hat x|^{-1}).
    \end{split}
    \end{align}
   \end{itemize}
\end{lemma}

\begin{proof}

Let us prove ~\eqref{itm:cantor-estimates1}. Let $l_j$ be the sequence of interval lengths defining the Cantor set $\frC_{\lambda,\gamma}$. Let $t \in (l_{j}/2 - l_{j+1}, l_{j-1}/2 - l_j)$. The set $\frC_{*,t}$ consists of $2^{mj}$  cubes with edge $l_j+2t$.
So
$$
|\frC_{*,t}|_m\leq 2^{m(j-1)} (l_j + 2t)^m.
$$

First consider the case $\frC=\frC^m_{\lambda,\gamma}$ with $\lambda>0$. Then $l_{j}+2t \leq l_{j-1}-l_j \leq c t$ with some constant $c$ independent of $j$, so 
$$
|\frC_t|_m \leq |\frC_{*,t}|_m \lesssim 2^{mj} t^m.
$$
Recalling that $\frD = -m\log 2 / \log \lambda$ and $\lambda^{\frD} = 2^{-m}$, we get 
$$
2^{mj} = \lambda^{-j\frD}  = l_j^{-\frD} j^{\gamma \frD} \eqsim l_j^{-\frD} \log^{\gamma \frD} (1/l_j) \eqsim t^{-\frD} \log^{\gamma \frD} (1/t).
$$

Now consider the case $\frC=\frC^m_{0,\gamma}$ (ultrathin Cantor sets). Then we get
$$
2^{mj}  = \log^{\gamma m} (1/l_j) \approx \log^{\gamma m} (1/t) .
$$

Let us prove ~\eqref{itm:cantor-estimates2}. We estimate $\mu^m_{\lambda,\gamma}(B^{\bar x}_r)$. Any interval of length $2t$ with $t\in (l_{j}/2 - l_{j+1}, l_{j-1}/2 - l_j)$ can intersect at most one interval forming the $j$-th iteration of the pre-Cantor set. Since $B^{\bar x}_t$  lies within a cube with edge $2t$, then $\mu_{\lambda,\gamma}^m(B^{\bar x}_t) \leq 2^{-jm}$. Using the above estimates for $2^{jm}$ we arrive at the required estimate.

 Let us prove~\eqref{itm:cantor-estimates4}. Note that
  \begin{align*}
    \lefteqn{\big((\mu_\lambda^m \times \delta_0^{d-m}) * (
    \indicator_{\set{\abs{\bar{y}}\leq \tau\abs{\hat{y}}}} \abs{y}^{-s}) \big)(x)}
    \qquad
    &
    \\
    &= \int\limits_{\setR^m}
      \indicator_{\set{\abs{\bar{x}-\bar{y}}\leq
      \tau
      \abs{\hat{x}}}} \abs{(\bar{x},\hat{x})-(\bar{y},0)}^{-s}\,d\mu_\lambda^m(\bar{y})
    \\
    &\leq \abs{\hat{x}}^{-s}  \int\limits_{\setR^m}
      \indicator_{\set{\abs{\bar{x}-\bar{y}}\leq
      \tau
      \abs{\hat{x}}}} \,d\mu_\lambda^m(\bar{y})
    \\
    &= \abs{\hat{x}}^{-s}
      \mu^m_{\lambda,\gamma} (B^m_{\tau \abs{\hat{x}}}(\bar{x})).
  \end{align*}
  Now, the claim follows by an application
  of the previous estimates. 
\end{proof}

Let $\Omega = (-1,1)^d$.

\begin{lemma}\label{L:integr_est}
For $\beta > 0$ and $t>3$ there holds 
$$
\mathcal{L}^d\{ |\hat x|^{-\beta} \log^\delta (e+|\hat x|^{-1}) \indicator_{\set{d(\bar x,\frC^m_{\lambda,\gamma}) \leq 4 |\hat x|}}>t \} \lesssim  t^{(\frD -d)/\beta} (\log t)^{\gamma \frD + (d-\frD)\delta/\beta}
$$
if $\lambda>0$ and 
$$
\mathcal{L}^d\{ |\hat x|^{-\beta}\log^\delta (e+|\hat x|^{-1}) \indicator_{\set{d(\bar x,\frC^m_{0,\gamma}) \leq 4 |\hat x|}}>t \} \lesssim  t^{-d/\beta} (\log t)^{\gamma m+d\delta/\beta}.
$$
In particular,
$$
|\hat x|^{-\beta} \log^\delta (e+|\hat x|^{-1})  \indicator_{\set{d(\bar x,\frC^m_{\lambda,\gamma}) \leq 4 |\hat x|}} \in L^1(\Omega)
$$
if $d>\frD + \beta$ or $d=\frD + \beta$ and additionally $\gamma\frD+ \delta  < -1$. Similarly, 
$$
|\hat x|^{-\beta} \log^\delta (e+|\hat x|^{-1}) \indicator_{\set{d(\bar x,\frC^m_{0,\gamma}) \leq 4 |\hat x|}} \in L^1(\Omega)
$$
if $d>\beta$ or $d=\beta$ and $\gamma m + \delta<-1$.

\end{lemma}

\begin{proof}
Consider the case $\frD>0$. If $|\hat x|^{-\beta} \log^\delta (e+|\hat x|^{-1}) >t$ then $
|\hat x|<ct^{-1/\beta} (\log t)^{\delta/\beta}$. We have 
\begin{align*}
\int\limits_{d(\bar x,\frC^m_{\lambda,\gamma}) \leq 4 |\hat x|< 4t^{-1/\beta}}\, dx = \int\limits_{|\hat x|<ct^{-1/\beta}(\log t)^{\delta/\beta}} \mathcal{L}^m (\{d(\bar x,\frC^m_{\lambda,\gamma}) \leq 4 |\hat x|\})\, d\hat x \\
\lesssim \int\limits_0^{c t^{-1/\beta}(\log t)^{\delta/\beta}} r^{d-m-1} r^{m-\frD} (\log r^{-1})^{\gamma \frD}\, dr
\lesssim t^{(\frD -d)/\beta} (\log t)^{\gamma \frD + (d-\frD)\delta/\beta}.
\end{align*}
The rest easily follows. 
\end{proof}

\section{General examples framework}

In this section we shall present a unified framework which reduces constructing an example for a particular model to checking integrability properties for a couple of functions.

\subsection{Competitor functions}

Our arguments will be based on comparing the energy of the minimizer with special competitor functions.

In the construction of our fractal examples we need a
smooth approximation of the indicator function
$\indicator_{\set{ d(\bar{x},\frC^{m}_\lambda) \leq 3
    \abs{\hat{x}}}}$, where
$x= (\bar{x},\hat{x}) \in \setR^m \times \setR^{d-m}$. This is the
purpose of the following lemma.

  Let~$1 \leq m \leq d-1$ and~$\frS := \frC^m_{\lambda,\gamma} \times \setR^{d-m}$. We use the
  notation $x=(\bar{x},\hat{x}) \in \setR^m \times \setR^{d-m}$. 
\begin{lemma}
  \label{lem:smooth-indicator}
 Let
  $\frac 14 \leq \tau_1 < \tau_2 \leq 4$ and~$\tau_2-\tau_1 \geq \frac
  14$. Then there exists~$\rho \in C^\infty(\setR^d \setminus \frS)$
  such that
  \begin{itemize}
  \item \label{itm:rho1}
    $\indicator_{\set{d(\bar{x},\frC^m_{\lambda,\gamma}) \leq \tau_1 \abs{\hat{x}}}}
    \leq \rho \leq \indicator_{\set{ d(\bar{x},\frC^m_{\lambda,\gamma}) \leq
        \tau_2 \abs{\hat{x}}}}$.
  \item \label{itm:rho2}
    $\abs{\nabla \rho(\bar{x},\hat{x})} \lesssim \abs{\hat{x}}^{-1}
    \indicator_{\set{\tau_1 \abs{\hat{x}} \leq d(\bar{x},\frC^m_{\lambda,\gamma}) \leq
        \tau_2 \abs{\hat{x}}}}$.
  \end{itemize}
  In particular, $\rho=1$ on~$\set{d(\bar{x},\frC^m_{\lambda,\gamma}) \leq \tau_1 \abs{\hat{x}}}$ and $\rho=0$ on $\set{\tau_2 \abs{\hat{x}} \leq d(\bar{x},\frC^m_{\lambda,\gamma})}$.
\end{lemma}

 For~$d \geq 2$ and $x\in \setR^d$ we denote $x=(\bar x,x_d)$, $\bar x \in \mathbb{R}^{d-1}$. Define $u_d$  on~$\setR^d$ by
$$
u_d(x) := \frac{1}{2}\sgn(x_d)\, \theta\bigg( \frac{\abs{x_d}}{\abs{\bar{x}}} \bigg).
$$  
where  $\theta\in C^\infty((0,\infty))$ such that $\indicator_{(1/2,\infty)}\leq \theta \leq \indicator_{(1/4,\infty)}$, $\|\theta'\|_\infty \leq 6$.

\begin{proposition}
  \label{pro:prop-uAb-d}
  There holds $u_d \in L^\infty(\setR^d) \cap W^{1,1}_{\loc}(\setR^d) \cap C^\infty(\setR^d
    \setminus \set{0})$.
  Moreover, the following estimates hold
  \begin{align*}
    \begin{alignedat}{2}
      \abs{\nabla u_d} &\lesssim \abs{x_d}^{-1} \indicator_{ \set{ 2
          \abs{x_d} \leq \abs{\bar{x}} \leq 4 \abs{x_d}}} &&\eqsim
      \abs{\bar{x}}^{-1} \indicator_{ \set{ 2 \abs{x_d} \leq
          \abs{\bar{x}} \leq 4 \abs{x_d}}}
     \end{alignedat}
  \end{align*}
\end{proposition}

Let $\Omega := (-1,1)^d$ with~$d\geq 2$. 

\begin{definition}[Competitors and auxiliary functions]
  \label{def:fractal-examples}
  Let $1 < p_0 < \infty$. We
  define~$u$ and $b$ on~$\overline{\Omega}$ distinguishing three cases:
  \begin{itemize}
  \item (Matching the dimension; Zhikov) $p_0=d$:

    Let the competitor function $u:= u_d$, the auxiliary function $b=|x_d|^{1-d}\indicator_\set{\abs{\bar x}<|x_d|/2}$,   $\frS := \set{0}$, ~$\frD    := \dim \frS =0$. 
    
  \item (Sub-dimensional) $1<p_0<d$:

    Let $\frS := \frC^{d-1}_{\lambda,\gamma} \times \set{0}$ and
    $\frD= \dim(\frS)=\frac{(d-1)\log 2}{\log(1/\lambda)}$,
    where~$\lambda \in (0,\frac 12)$ is chosen such that
    $p_0 = d - \frD$.  Let $\rho \in C^\infty(\setR^d\setminus \frS)$
    be such that (using Lemma~\ref{lem:smooth-indicator})
    \begin{align*}
    \indicator_{\set{d(\bar{x},\frC^{d-1}_{\lambda,\gamma}) \leq 2 \abs{x_d}}}
      &\leq \rho \leq \indicator_{\set{d(\bar{x},\frC^{d-1}_{\lambda,\gamma})
          \leq 4 \abs{x_d}}},\\
      \abs{\nabla \rho}& \lesssim \abs{x_d}^{-1}
      \indicator_{\set{2 \abs{x_d} \leq d(\bar{x},\frC^{d-1}_{\lambda,\gamma}) \leq
          4 \abs{x_d}}}.
    \end{align*}
    We define the competitor $u:= \frac{1}{2}\sgn(x_d)\, \rho(x)$. Let $\nu=\frD$ if $\frD>0$ and $\nu = d-1$ if $\frD =0$. Introduce the auxiliary function
   \begin{equation}\label{bdef_sub}
  b := \indicator_{ \set{ 
          d(\bar{x},\frC^{d-1}_{\lambda,\gamma}) \leq \frac 12 \abs{x_d}}} \abs{x_d}^{\frD+1-d} \log^{-\gamma \nu} (e+\abs{x_d}^{-1}). 
\end{equation}
If in this construction $\frD=\lambda=0$, we denote this case by $p_0=d-0$ and include it in the subcritical framework.

  \item (Super-dimensional) $p_0 > d$:
    
    Let $\frS := \set{0}^{d-1} \times \frC_{\lambda,\gamma}$ and
    $\frD= \dim(\frS)=\frac{\log 2}{\log(1/\lambda)}$,
    where~$\lambda \in (0,\frac 12)$ is chosen such that
    $p_0 = \frac{d-\frD}{1-\frD}$. We define the competitor
    $$
    u:= (\delta_0^{d-1} \times \mu_{\lambda,\gamma}) * u_d \Leftrightarrow u(\bar x,x_d)= \int u_d(\bar x, x_d-y_d)d\mu_{\lambda,\gamma}(y_d).
    $$  
 Introduce the auxiliary function 
\begin{equation}\label{bdef_super}
b= \abs{\bar x}^{1-d} \indicator_\set{2\abs{\bar x} \leq d (x_d, \frC_{\lambda,\gamma})\leq 4 \abs{\bar x}}.
\end{equation}
If in this construction $\frD=\lambda=0$, we denote this case by $p_0=d+0$ and include it in the supercritical framework.

  \end{itemize}
\end{definition}


Using Lemma~\ref{lem:cantor-estimates} (estimate \eqref{itm:cantor-estimates4}) and the structure of $u$ we get the following result. 
\begin{proposition}
  \label{pro:est-uAb}
  For~$1<p_0<\infty$ let $u$ be as is
  Definition~\ref{def:fractal-examples}.  Then $u \in L^\infty(\Omega)  \cap C^\infty(\overline{\Omega}
  \setminus \frS)$ and the following holds.
  
  \begin{itemize}
\item  \noindent  If $p_0=d$, then
$$
    \abs{\nabla u} 
    \lesssim           \abs{\bar{x}}^{-1} \indicator_{ 
      \set{ 2 \abs{x_d} \leq \abs{\bar{x}} \leq 4
      \abs{x_d}}} \eqsim \abs{x_d}^{-1} \indicator_{ 
      \set{ 2 \abs{x_d} \leq \abs{\bar{x}} \leq 4
      \abs{x_d}}}.
$$

\item If $1 < p_0 < d$ and $\lambda>0$ or $p_0=d-0$, then
\begin{equation}\label{du_sub}   
    \abs{\nabla u} \lesssim \abs{x_d}^{-1} \indicator_{
                     \set{
                     2 \abs{x_d} \leq 
                     d(\bar{x},\frC_{\lambda,\gamma}^{d-1}) \leq 4 \abs{x_d}
                     }
                     }
\end{equation}

\item If $p_0>d$ and $\lambda>0$ or $p_0=d+0$, then
\begin{gather}\label{du_super}
    \abs{\nabla u} \lesssim \abs{\bar{x}}^{\frD-1} \log^{-\gamma \frD} (e+|\bar x|^{-1}) \indicator_{ \set{ 
                     d(x_d,\frC_{\lambda,\gamma}) \leq \frac 12 \abs{\bar{x}}}}
                     , \quad \lambda>0,\\   
\label{du_super1}   \abs{\nabla u} \lesssim \abs{\bar{x}}^{-1} \log^{-\gamma}(e+ |\bar x|^{-1}) \indicator_{ \set{ 
                     d(x_d,\frC_{0,\lambda}) \leq \frac 12 \abs{\bar{x}}}}.
\end{gather}

  \end{itemize}
\end{proposition}

\begin{lemma}\label{L:u}
There holds $u\in W^{1,1}(\Omega)$. 
\begin{itemize}

\item There holds $\nabla u \in L^{p_0-\varepsilon}(\Omega)$ for any $\varepsilon>0$. 

\item If $\gamma=0$ there holds $\nabla u \in L^{p_0,\infty}(\Omega)$.

\item In the subcritical construction $\nabla u \in L^{p_0}(\Omega)$ provided that  $\gamma<-1/\frD$ if $p_0<d$ or $\gamma<-1/(d-1)$ and $p_0=d-0$.

\item In the supercritical construction, $\nabla u \in L^{p_0}(\Omega)$ provided that $\gamma\frD (p_0-1)>1$ if $p_0>d$ or if $\gamma (d-1)>1$ and $p_0=d+0$.

\end{itemize}
\end{lemma}
\begin{proof}

The first claim follows from estimates \eqref{du_sub}, \eqref{du_super} and Lemma~\ref{L:integr_est}: in the subdimensional case for $\beta=p<p_0$ we have $d>\beta + \frD $ and in the superdimensional case for $\beta = (1-\frD)p$ with $p<p_0$ we also have $d>\beta+\frD$. In particular, $\nabla u\in L^1(\Omega)$. Since $u$ is absolutely continuous on almost all lines this also implies that $u\in W^{1,1}(\Omega)$.

For the second claim in the subdimensional case  we use estimate \eqref{du_sub} and Lemma~\ref{L:integr_est} with $\beta = p_0=d-\frD$, $\delta=0$, and in the superdimensional case we use estimate \eqref{du_super} and Lemma~\ref{L:integr_est} with $\beta = (1-\frD) p_0 = d-\frD$, $\delta=0$.

The third claim follows from estimate \eqref{du_sub} and Lemma~\ref{L:integr_est} with $\beta = p_0= d-\frD$ and $\delta =0$.

The fourth claim follows from estimate \eqref{du_super}, \eqref{du_super1} and Lemma~\ref{L:integr_est} with $\beta =(1-\frD)p_0= d-\frD$ and $\delta = -\gamma \frD p_0$ if $\frD>0$ and $\delta = -\gamma p_0$ if $\frD=0$.
\end{proof}

In a similar way we obtain the following statement.
\begin{lemma}\label{L:b}
There holds $b\in L^{s}(\Omega)$ for any $s<p_0'$.
\end{lemma}
\begin{proof}

For $p_0<d$ or $p_0=d-0$ use \eqref{bdef_sub} and Lemma~\ref{L:integr_est} with $\beta = (d-1-\frD)s=(p_0-1)s$ and $\delta = -\gamma \mu s$. It remains to note that for $s<p_0'$ there holds $(p_0-1)s< p_0=d-\frD$. 

For $p_0>d$ or $p_0=d+0$ use \eqref{bdef_super} and Lemma~\ref{L:integr_est} with $\beta = s(d-1)$ and note that for $s< p_0' = (d-\frD)(d-1)^{-1}$ we have $\beta < d-\frD$.
\end{proof}

\subsection{Main assumption and conditional results}\label{ssec:condition}

We shall first state conditional results. Let $1<p_0<\infty$, and $u$ and $b$ be the competitor and auxiliary functions constructed in Definition~\ref{def:fractal-examples} using the Cantor set $\frC^{d-1}_{\lambda,\gamma}$. Let $\Omega = (-1,1)^d$ and
$$
\mathcal{F}(w) = \int\limits_\Omega \phi(x,\abs{\nabla w})\,dx, \quad \mathcal{F}^*(f) = \int\limits_\Omega \phi^*(x,\abs{f})\,dx.
$$

Our main assumption will be the following.

\begin{assumption}
  \label{ass:harmonic2}
  There holds 
\begin{equation}\label{MC1}
\mathcal{F}(u),\mathcal{F}^*(b) < \infty.
\end{equation}
Moreover, for every $\kappa> 0$ there exists $s,t> 0$ such that
  \begin{equation}\label{MC2}
    \mathcal{F}(\eta u) + \mathcal{F}^*(sb) < \kappa\, \eta s.
  \end{equation}
\end{assumption}

We shall also use a weaker form of this assumption. Let $\kappa>0$ be given. 
\begin{assumption}\label{ass:harmonic3} (\textit{with} $\kappa$)
There holds \eqref{MC1} and for a given $\kappa$ there exists $s,t>0$ such that \eqref{MC2} is satisfied.
\end{assumption}
Assumption~\ref{ass:harmonic3} depends on the given parameter $\kappa$.

Note that the results of this section are independent of the particular form of the integrand and rely only on the Assumption~\ref{ass:harmonic2} (or Assumption~\ref{ass:harmonic3} with sufficiently small $\kappa$). Further we shall verify this assumption for particular models.

For $\eta \geq 0$ let $v_\eta$ denote the $W^{1,\phi(\cdot)}(\Omega)$-minimizer of~$\mathcal{F}$ with boundary values $v_\eta = \eta u$:
\begin{equation}\label{vtdef}
v_\eta = \mathrm{arg\,min} \mathcal{F}(\eta u+ W_0^{1,\phi(\cdot)}(\Omega)).
\end{equation}
In the subdimensional case $p_0<d$ or $p_0=d$ we denote the set of non-Lebesgue points of $v_\eta$ by $\Sigma_\eta$. Let $\Sigma'_\eta = \Sigma_\eta \cap ((-1,1)^{d-1}\times \{0\})$. Recall that in the subdimensional case we take $p_0=d-\frD$, $\frD = (d-1) \frac{\log 2}{\log \lambda^{-1}}$, that is $\lambda = 2^\frac{1-d}{d-p_0}$.

\begin{theorem}\label{T:condB}
If $p_0<d$ or $p_0=d-0$, then under Assumption~\ref{ass:harmonic2} (or Assumption~\ref{ass:harmonic3} with sufficiently small $\kappa=\kappa(d,\lambda,\gamma)$) for some $\eta>0$ we have $\mu_{\lambda,\gamma}^{d-1} (\Sigma'_\eta)>\frac 12 $. 
\end{theorem}

Theorem~\ref{T:condB} follows from Theorem~\ref{T:diff_trace} proved in the next section. By this result from the Frostman lemma and the regularity of the measure~$\mu$ we derive the following 
\begin{corollary}\label{cor:regul}
In the sub-critical case $1<p_0<d$, under Assumption~\ref{ass:harmonic2} (or Assumption~\ref{ass:harmonic3} with sufficiently small $\kappa=\kappa(d,\lambda,\gamma)$) for some $\eta>0$ there exists a closed  set with  Hausdorff dimension  $\mathrm{dim}_\mathcal{H} \Sigma =\frD$, such that all points of this set are non-Lebesgue points of the minimizer~$v_\eta$. 
\end{corollary}

\begin{theorem}\label{T:condA}
Under Assumption~\ref{ass:harmonic2} (or Assumption~\ref{ass:harmonic3} with sufficiently small $\kappa=\kappa(d,\lambda,\gamma)$) the $W$-minimizer $v_\eta\notin W^{1,p_0+\varepsilon}(\Omega)$ for any $\varepsilon>0$.
\end{theorem}

The proof of Theorem~\ref{T:condA} is divided between Corollary~\ref{corr_integr0} in the subdimensional case and Corollary~\ref{corr:integ1} in the superdimensional case. One can get more precise results using Theorems~\ref{T:absenceHI0}, \ref{T:abssuper}.

Recall that in the superdimensional case $p_0\geq d$ we set $p_0=\frac{d-\frD}{1-\frD}$, $\frD=\frac{\log 2}{\log \lambda^{-1}}$, $\nu=1$ if $p_0=d+0$, and $\nu = \frD$ if $p_0>d$
\begin{theorem}\label{T:condC}
If $p_0=\frac{d-\frD}{1-\frD}$ then under Assumption~\ref{ass:harmonic2} (or Assumption~\ref{ass:harmonic3} with sufficiently small $\kappa$) for some $\eta>0$ the modulus of continuity of the minimizer $v_\eta$ is not better than $\omega(\rho)=C\rho^\frD \log^{-\gamma \nu} (e+\rho^{-1})$ in the sense that it cannot be replaced by any function $\omega_1(\rho)$ such that $\omega_1(\rho) = o(\omega(\rho))$ as $\rho \to 0$.
\end{theorem} 

\section{Irregularity of minimizers: subdimensional case}\label{sec:sub}

In this section we state a conditional result which claims that the set of discontinuity of the minimizer is roughly speaking almost as large as our contact set $\frS = \frC^{d-1}_{\lambda,\gamma} \times\{0\}$. From this fact we derive non-improvement of the integrability of the gradient of the minimizer.

\subsection{Restricted Riesz potentials --- subdimensional case}\label{ssec:rrp}

Let $y=(\bar y, y_d)$ and $C = \{ \abs{\bar{y}} \leq \frac 12 \abs{y_d} \}$, $C^{\pm} = C \cap \{\pm y_d>0\}$. Denote $C^{\pm}(x) = \{x+y:\, y\in C^\pm\}$, $C(x)=\{x+y:\, y\in C\}$. 

Define the restricted Riesz potentials $I^\pm_1(f)(x)$ and $I_1^{\btimes} f(x)$ by
\begin{align*}
    I^\pm_1(f)(x):=  \int\limits_{C^{\pm}(x)}\frac{\abs{f(y)}}{\abs{x-y}^{d-1}}\, dy= \int\limits_{\setR^d}\frac{f(x-y)}{\abs{y}^{d-1}} \indicator_{\set{\abs{\bar{y}} \leq \pm\frac 12 \abs{y_d}}}\, dy,\\ 
I_1^{\btimes} f(x)= I_1^+ f(\bar{x},0)  + I_1^- f(\bar{x},0) \\
= \int\limits_{C(x)}\frac{\abs{f(y)}}{\abs{x-y}^{d-1}}\, dy=\int\limits_{\setR^d}\frac{f(x-y)}{\abs{y}^{d-1}} \indicator_{\set{\abs{\bar{y}} \leq \frac 12 \abs{y_d}}}\, dy.
\end{align*}


Recall that in the subdimensional case we set $\frS = \frC^{d-1}_{\lambda,\gamma} \times\{0\}$. 
Denote 
$$
\mathcal{M}^+=\bigcup\limits_{x\in \frS} C(x), \quad \mathcal{M}^{-}=\Omega \setminus \mathcal{M}^+.
$$

Let the function $b$ be from Definition~\ref{def:fractal-examples}. Then $\mathrm{supp}\, b \in \mathcal{M}^+$.


\begin{lemma}
  \label{lem:restricted-riesz}
 For a function $f$ with support in $\overline \Omega $ there holds
  \begin{align*}
    \int\limits_{(-1,1)^{d-1}} I_1^{\btimes}(\abs{f})(\bar x,0)d\mu^{d-1}_{\lambda,\gamma}(\bar x)
    &\leq C(d,\lambda,\gamma)\, \int\limits_{\Omega} \abs{f(y)} \cdot
     b\,dy,
  \end{align*}

\end{lemma}
\begin{proof}
  Let $x := (\bar{x},0)$. Then with $y=(\bar{y},y_d)$ by Fubini's theorem we have
  \begin{align*}
    \lefteqn{\int\limits_{(-1,1)^{d-1}} I_1^{\btimes}(\abs{f})(\bar x,0)\,d\mu^{d-1}_{\lambda,\gamma}(\bar
    x) } \qquad
    \\
    &=\int\limits_{(-1,1)^{d-1}} \int \frac{\abs{f(y)}}{\abs{x-y}^{d-1}
      }
      \indicator_{\set{\abs{\bar{x}-\bar{y}}< \frac 12\abs{y_d}}}\,dy\, d\mu^{d-1}_{\lambda,\gamma}(\bar{x})
    \\
    &=\int\limits_{\setR^d} |f(y)| \bigg(\underbrace{\int\limits_{(-1,1)^{d-1}} \frac{1}{\abs{x-y}^{d-1}
      } \indicator_{\set{\abs{\bar{x}-\bar{y}}< \frac 12 \abs{y_d}}}\,d\mu^{d-1}_{\lambda,\gamma}(\bar x)}_{=:K(y)} \Bigg) \, dy.
  \end{align*}
  The estimate on $K(y)$ by Lemma~\ref{lem:cantor-estimates} yields  $K \lesssim b$.
\end{proof}

\subsection{Restricted Riesz potential for minimizers}

Let $Q\geq 1$. Let $w_\eta$ be a function such that $w_\eta\in \eta u + W_0^{1,\phi(\cdot)}(\Omega)$ and $\mathcal{F}(w_\eta)\leq Q \mathcal{F}(\eta u)$. Mainly we shall use  $w_\eta=v_\eta$ (recall that by $v_\eta$ we denote the minimizer of $\mathcal{F}$ with the boundary value $u\eta$). This corresponds to $Q=1$.

Denote
$$
J_0(\eta)= \int\limits_{(-1,1)^d} \abs{\nabla w_\eta(y)}\, b\,dy. 
$$

\begin{lemma}
  \label{lem:riesz-sublinear00}
Under Assumption~\ref{ass:harmonic2}, for every $\eta>0$ the integral $J_0(\eta)$ is finite. Moreover, for every $\kappa>0$ there exists $\eta>0$ such that $J_0(\eta) \leq \kappa \,\eta$.
\end{lemma}
\begin{proof}
  For all $\eta,s>0$
  $$
    J_0(\eta) \leq \frac{1}{s} \big(\mathcal{F}(w_\eta) + \mathcal{F}^*(s b)\big)
  $$
  Using that $\mathcal{F}(w_eta)\leq Q \mathcal{F}(\eta u)$ for all $\eta,s>0$ we obtain   
  $$
    J_0(\eta)\leq \frac{Q}{s} \big(\mathcal{F}(\eta u) + \mathcal{F}^*(s b)\big).
  $$
  Due to Assumption~\ref{ass:harmonic2} we can find and fix~$\eta$ and
  $s$ such that
  $$
    \mathcal{F}(\eta u) + \mathcal{F}^*(s b) < \kappa\, \eta s Q^{-1}.
  $$
  This proves as desired~$J_0(\eta) \leq \kappa\,\eta$.
\end{proof}

The same argument also gives
\begin{lemma}\label{lem:riesz-sublinear01}
Under Assumption~\ref{ass:harmonic3} the integral $J_0(\eta)$ is finite. For every $\kappa'>0$ there exists such $\kappa>0$ that  Assumption~\ref{ass:harmonic3} with this $\kappa$ gives $J_0(\eta)\leq \kappa' \eta$. 
\end{lemma}
Note that in Lemma~\ref{lem:riesz-sublinear01} the number $\eta$ is not arbitrary, it comes from Assumption~\ref{ass:harmonic3} with specific $\kappa$.

\begin{lemma}\label{lem:riesz-sublinear}
Under Assumption~\ref{ass:harmonic2} or Assumption~\ref{ass:harmonic3} the integral 
$$
J(t)=\int\limits_{(-1,1)^{d-1}} I_1^{\btimes}(\indicator_{(-1,1)^d} \abs{\nabla w_\eta})(\bar x,0)\,d\mu^{d-1}_{\lambda,\gamma}(\bar x)<\infty.
$$
Under Assumption~\ref{ass:harmonic2} for every $\kappa>0$ there exists $\eta>0$ such that $J(\eta) \leq \kappa \,t$. For every $\kappa'>0$ there exists such $\kappa$ that  Assumption~\ref{ass:harmonic3} with this $\kappa$ gives $J(\eta)\leq \kappa' \eta$. 
\end{lemma}
\begin{proof}
Follows immediately from Lemmas~\ref{lem:restricted-riesz}, \ref{lem:riesz-sublinear00}, \ref{lem:riesz-sublinear01}.
\end{proof}

\subsection{Trace estimates}\label{ssec:trace}

Let $\omega \in L^\infty(B^d_1(0))$, $\omega\geq 0$, $\mathrm{supp}\, \omega \subset \{|x_d|<1/2\}\cap B_1^d(0)$, and $\int\limits_{B^d_1(0)} \omega\, dx=1$. Define 
$$
\omega^\pm_r(y) = r^{-d}\omega(r^{-1}(\bar y - \bar x), r^{-1}y_d \mp 1), \quad \widetilde{B}^{\pm}_r (\bar x) = B^d_{r} (\bar x,\pm r).
$$
We have $\mathrm{supp}\, \omega^\pm_r \subset \widetilde{B}^{\pm}_r (\bar x)$. Denote 
$$
D^{\pm}_{r_1,r_2}=\mathrm{conv}\, (\mathrm{supp}\, \omega^\pm_{r_1}\cup \mathrm{supp}\, \omega^\pm_{r_2})
$$
For $\bar x\in (-1,1)^{d-1}$, $r>0$ and a function $f$ define

$$
\langle  f\rangle^\pm_{\bar x, r} =\int\limits_{B_1^d(0)} \omega(y) f(\bar x+r \bar y, y_d r\pm r)\, dy =\int\limits_{\widetilde{B}^{\pm}_r (\bar x)} f(z)\omega_r(z)\, dz.  
$$
Let $0<r_1<r_2$. For $f\in W^{1,1} (D^+_{r_1,r_2})$ define the function $\widetilde f(y,s) =\omega(y) f(\bar x+s \bar y, y_d s+s)$.  We have 
\begin{gather*}
\langle  f\rangle^+_{\bar x,r_2} - \langle  f\rangle^+_{\bar x ,r_1} 
=\int\limits_{B^{d}_1(0)} (\widetilde{f} (y,r_2) - \widetilde{f}(y,r_1))\, dy\\
=\int\limits_{B^{d}_1(0)} \int\limits_{r_1}^{r_2} \omega(y)(\bar y, y_d+1)\cdot \nabla f (\bar x+s \bar y, y_d s+s)\, dy\, ds\\
=\int\limits_{D^+_{r_1,r_2}} K^{+}(\bar x, z, r_1,r_2)\nabla f(z)\, dz,\\
K^+(\bar x, z, r_1,r_2) = (z- (\bar x,0)) \int\limits_{r_1}^{r_2} \omega (s^{-1}(\bar z-\bar x), s^{-1}z_d-1) s^{-1-d} \, ds.
\end{gather*}
If $\omega (s^{-1}(\bar z-\bar x), s^{-1}z_d-1)\neq 0$  then $2z_d/3\leq s\leq 2z_d$ and $|\bar z - \bar x|\leq 2z_d$. Thus 
$$
|K^+(\bar x, z, r_1,r_2)| \leq C(d)\|\omega\|_\infty \indicator_{D_{r_1,r_2}} z_d^{1-d}.
$$
Thus for $0<r_1<r_2$ we have 
\begin{gather*}
|\langle  f\rangle^+_{\bar x, r_2} - \langle  f\rangle^+_{\bar x,r_1}| \leq C(d)\|\omega\|_\infty\int  \indicator_{D^+_{r_1,r_2}} (z) |\nabla f(z)| z_d^{1-d}\, dz\\
\leq C(d)\|\omega\|_\infty \int\limits_{C^{+}(\bar x,0)} \indicator_\set{r_1/2\leq z_d\leq 2r_2}(z) |\nabla f(z)| z_d^{1-d}\, dz.
\end{gather*}
Similarly, for $0>r_1>r_2$ we hat 
\begin{gather*}
|\langle  f\rangle^-_{\bar x,r_2} - \langle  f\rangle^-_{\bar x ,r_1}| \leq C(d)\|\omega\|_\infty\int  \indicator_{D^-_{|r_1|,|r_2|}} (z) |\nabla f(z)| z_d^{1-d}\, dz\\
\leq C(d)\|\omega\|_\infty\int\limits_{C^{-}(\bar x,0)} \indicator_\set{2r_2 <z_d<r_1/2}(z) |\nabla f(z)| z_d^{1-d}\, dz.
\end{gather*}
Let $\mu$ be the Cantor measure corresponding to the given fractal set $\frC=\frC^{d-1}_{\lambda,\gamma}$.
\begin{corollary}\label{corr_conv}
For any point $\bar x\in (-1,1)^{d-1}$ where $I_1^{\btimes}[\nabla f] (\bar x, 0)$ is finite there exist the limits (upper and lower traces)
\begin{equation}\label{conv}
f_+(\bar x) := \lim_{r\to 0+} \langle f \rangle^+_{\bar x,r} , \quad  f_{-}(\bar x) := \lim_{r\to 0-} \langle  f\rangle^-_{\bar x,r}. 
\end{equation} 
If moreover 
$$
\int\limits_{(-1,1)^{d-1}} I_1^{\btimes}(|\nabla f|\indicator_{(-1,1)^d})(\bar x,0) d\mu(\bar x)< \infty
$$ 
then convergence in \eqref{conv} is in the sense of $L^1(\mu)$ and the upper and lower traces $f_{\pm}$ exist in the sense of $L^1((-1,1)^{d-1},\mu)$.
\end{corollary}

We shall use these estimates for $\omega(y) = |B_\tau^d(0)|^{-1}\indicator_{B_\tau^d(x)}(y)$, $x\in B^d_1(0)$.

We also state without proof the following elementary inequality. Let $r_1,r_2>0$, $x,z\in \mathbb{R}^d$.Then
\begin{gather*}
\biggl| \frac{1}{|B^d_{r_1}(x)|}\int\limits_{B^d_{r_1}(x)} f\, dy - \frac{1}{|B^d_{r_2}(z)|}\int\limits_{B^d_{r_1}(z)} f\, dy \biggr| \\
\leq C(d) \frac{|x-z|+|r_1-r_2|}{\min(r_1,r_2)^d} \int\limits_{\mathrm{conv} (B^d_{r_1}(x)\cup B^d_{r_2}(z))}|\nabla f|\, dy.
\end{gather*}

\subsection{Discontinuity of minimizers}\label{ssec:subdis}

In the following statement we show the discontinuity of the minimizer on the set of points comparable in Cantor measure to the whole contact set $\frS=\frC\times \{0\}$, $\frC = \frC^{d-1}_{\lambda,\gamma}$. Define upper and lower traces of the minimizer $(v_\eta)_{+}$ and $(v_\eta)_{-}$ on the contact set $\frS$ as in Section~\ref{ssec:trace} using $\omega(y) = |B_{1/4}^d(0)|^{-1}\indicator_{B_{1/4}^d(0)}(y)$. That is, 
$$
(v_\eta)_{\pm} (\bar x) = \lim_{r\to 0}\langle v_\eta \rangle^{\pm}_{\bar x,r},\quad  \langle v_\eta \rangle^{\pm}_{\bar x,r}=\frac{1}{|B^d_{r/4}(\bar x,\pm r)|}\int\limits_{B^d_{r/4}(\bar x,\pm r)} v_\eta(y)\, dy,
$$
with the convergence taking place for any $\bar x$ where $I_1^{\btimes}[\nabla v_\eta] (\bar x, 0)<\infty$, and in $L^1(\mu)$ provided that $\int I_1^{\btimes}[\nabla v_\eta] (\bar x, 0)\, d\mu<\infty$.

\begin{theorem}\label{T:diff_trace}
Under Assumption~\ref{ass:harmonic2}, for any $N>3$ there exists $\eta>0$ such that
\begin{equation}\label{mu_est}
\begin{gathered}
\mu^{d-1}_{\lambda,\gamma} (\{ |(v_\eta)_{+} - (v_\eta)_{-}|>\eta(1-N^{-1})\}) \geq 1- N^{-1},\\
\int\limits_{\frC} |(v_\eta)_+ (\bar x) - (v_\eta)_+ (\bar x)|\, d\mu^{d-1}_{\lambda,\gamma} (\bar x) \geq \eta(1-N^{-1}).
\end{gathered}
\end{equation}
\end{theorem}

\begin{proof}
Extend the function $v_\eta$ by $\pm \eta$ to $\{|\bar x|_\infty<\abs{x_d},\  \pm x_d>1\}$. For any point $\bar x\in \frC$ with finite $I_1^{\pm}[\nabla v_\eta] (\bar x, 0)$, we have 
$$
|(v_\eta)_\pm (\bar x) \mp \eta/2 | \leq C(d) I_1^{\pm}[|\nabla v_\eta| \indicator_\Omega] (\bar x, 0).
$$
Therefore, 
\begin{equation}\label{diff0}
\int\limits_{\frC} |(v_\eta)_\pm (\bar x) \mp \eta/2 |\, d\mu^{d-1}_{\lambda,\gamma}  (\bar x) \leq C(d) \int\limits_{(-1,1)^{d-1}} I_1^{\pm}[|\nabla v_\eta|\indicator_\Omega] (\bar x, 0)\, d \mu^{d-1}_{\lambda,\gamma}  (\bar x):=C(d) J(\eta).
\end{equation}
By the triangle inequality,
\begin{equation}\label{triangle}
\eta = |\eta/2-(-\eta/2)| \leq |\eta/2-(v_\eta)_{+}(\bar x)| +|(v_\eta)_{+}(\bar x)-(v_\eta)_{-}(\bar x)| +  |(v_\eta)_{-}(\bar x)-(-\eta/2)|  
\end{equation}
Integrating with respect to the Cantor measure and using \eqref{diff0} we get 
\begin{equation}\label{diff1}
\eta \leq 2C(d) J(\eta) + \int\limits_{\frC} |(v_\eta)_+ (\bar x) - (v_\eta)_+ (\bar x)|\, d\mu^{d-1}_{\lambda,\gamma} (\bar x).
\end{equation}
By Assumption~\ref{ass:harmonic2} and Lemma~\ref{lem:riesz-sublinear} (applied to $w_\eta=v_\eta$) for any $\kappa>0$ we can find $\eta>0$ such that 
\begin{equation}\label{diff_fpm}
\int\limits_{\frC} |(v_\eta)_+ (\bar x) - (v_\eta)_+ (\bar x)|\, d\mu^{d-1}_{\lambda,\gamma} (\bar x) \geq \eta- \kappa \eta.
\end{equation}
On the other hand, by the maximum principle 
$$
|(v_\eta)_+ (\bar x) - (v_\eta)_- (\bar x)| \leq \eta.
$$
For any $\alpha \in (0,1)$ we have
\begin{align*}
\int\limits_{\frC} |(v_\eta)_+ (\bar x) - (v_\eta)_- (\bar x)|\, d\mu^{d-1}_{\lambda,\gamma} (\bar x) \\
\leq \eta \cdot\mu^{d-1}_{\lambda,\gamma} (\{ |(v_\eta)_{+} - (v_\eta)_{-}|>\eta \alpha\}) + \alpha \eta \cdot(1-\mu^{d-1}_{\lambda,\gamma} (\{ |(v_\eta)_{+} - (v_\eta)_{-}|>\eta \alpha \})).
\end{align*}
Combining this with the estimate of this integral from below, we obtain
$$
(1-\alpha)\mu^{d-1}_{\lambda,\gamma} (\{ |(v_\eta)_{+} - (v_\eta)_{-}|>\eta\alpha\}) + \alpha \geq 1-\kappa.
$$
Let $\eta$ be such that additionally $\kappa<N^{-2}$. Then taking $\alpha = 1-N \kappa = 1-N^{-1}$ we get
$$
\mu^{d-1}_{\lambda,\gamma} (\{ |(v_\eta)_{+} - (v_\eta)_{-}|>\eta \alpha\}) \geq 1- \frac{\kappa}{1-\alpha} =1 -N^{-1},
$$
which is the required claim.
\end{proof}

Previous theorem can be also stated as follows.
\begin{theorem}\label{T:diff_trace1}
For any $N>3$ there exists $\kappa=\kappa(N,d,\lambda,\gamma)>0$ such that Assumption~\ref{ass:harmonic3} with this $\kappa$ yields \eqref{mu_est}.
\end{theorem}

\subsection{Absence of higher gradient integrability}

One can derive the absence of the higher gradient integrability from the usual characterization of the set of singular point of a Sobolev function: if $W^{1,s}(\Omega)$, $1\leq s <n$, then the Hausdorff dimension of set of its non-Lebesgue points is at most $n-s$.

However we prefer to give a direct proof since the similar argument will be also applied in the superdimensional case.

Let $\frC = \frC^{d-1}_{\lambda,\gamma}$. For $h>0$ by $\frC_h$ we denote the $h$-neighbourhood of $\frC$ in $(-1,1)^{d-1}$ and $C_h = \frC_{h/2}\times (-2h,2h)$. By Lemma~\ref{lem:cantor-estimates} (see \eqref{itm:cantor-estimates2}) the Lebesgue measure of $C_h$ goes to zero as $h\to 0$.

Recall that in the subdimensional case we denote $\nu=\frD$ if $\frD>0$ and $\nu = d-1$ if $\frD=0$.
\begin{theorem}\label{T:absenceHI0}
Let $\Psi$ be a generalized Orlicz function. If 
$$
G(\Psi,h):= \int\limits_{C_h} \Psi^* (x,h^{1-d+\frD} \log^{-\gamma \nu} (1/h))\, dx \to 0
$$
as $h\to 0$ then under Assumption~\ref{ass:harmonic2} (or Assumption~\ref{ass:harmonic3} with sufficiently small $\kappa$) there holds
$$
\int\limits_{\Omega} \Psi (x,|\nabla v_\eta|)=\infty
$$ 
for some $\eta>0$.
\end{theorem}
\begin{proof}

By Theorem~\ref{T:diff_trace} and Corollary~\ref{corr_conv}, there exists $\eta>0$ such that for sufficiently small $h$ there holds 
$$
\int |\langle v_\eta\rangle^+_{\bar x,h}-\langle v_\eta\rangle^-_{\bar x,h}| \, d\mu^{d-1}_{\lambda,\gamma}(\bar x)  \geq \frac{\eta}{2}.
$$
There holds
$$
|\langle v_\eta\rangle^+_{\bar x,h}-\langle v_\eta\rangle^-_{\bar x,h}| \leq C(d)h^{1-d} \int\limits_{B^{d-1}_{h/2}(\bar x) \times (-2h,2h)} |\nabla v_\eta|\, dy.
$$
Integrating this against the Cantor measure corresponding to the given Cantor set $\frC =\frC^{d-1}_{\lambda,\gamma}$ we get 
$$
C(d)\int\limits_{|y_d|\leq 2h} |\nabla v_\eta(y)| h^{1-d}\mu^{d-1}_{\lambda,\gamma}(B^{d-1}_{h/2}(\bar y))\, dy \geq \frac{\eta}{2}.
$$

Assume that $|\nabla v_\eta| \in L^{\Psi}(\Omega)$. By the generalized Young inequality, and using estimates of Lemma~\ref{lem:cantor-estimates}, \eqref{itm:cantor-estimates1}, we get
\begin{align*}
\int\limits_{|y_d|\leq 2h} |\nabla v_\eta(y)| h^{1-d}\mu^{d-1}_{\lambda,\gamma}(B^{d-1}_{h/2}(\bar y))\, dy  \\
\lesssim \int\limits_{C_h} \Psi(y,|\nabla v_\eta|)\, dy + \int\limits_{C_h} \Psi^* (y,h^{1-d+\frD} \log^{-\gamma \nu} (1/h))\, dy.
\end{align*}
This gives a contradiction as $h\to 0$.
\end{proof}

\begin{corollary}[Lack of Meyers property]\label{corr_integr0}
If $p_0<d$ or $p_0=d-0$, then under Assumption~\ref{ass:harmonic2} for some $\eta>0$  there holds 
\begin{align*}
 \int\limits_{\Omega}\abs{\nabla v_\eta}^{p_0}\log^{\delta}(e+\abs{\nabla v_\eta} ) dx=\infty
\end{align*}
for any~$\delta>-\gamma \nu$.
In particular $|\nabla v_\eta|\notin L^{s}(\Omega)$ for any $s>p_{0}$. The same holds under Assumption~\ref{ass:harmonic3} with sufficiently small $\kappa>0$.
\end{corollary}

\begin{proof}
Let  $\Psi(\tau)=\tau^{p_0}\log^\delta(e+\tau)$. Then
\begin{align*}
 \Psi^*(\tau)\lesssim \tau^{p_{0}'} \log^{\frac{\delta}{1-p_0}}(e+\tau). 
\end{align*}

By estimates of Lemma~\ref{lem:cantor-estimates} we have 
\begin{align*}
 \abs{C_h}\lesssim h^{d-\frD}\log^{\gamma\nu} (1/h),
\end{align*}
then
\begin{align*}
G(\Psi,h)& \lesssim \int\limits_{C_h} (h^{1-d+\frD}\log^{-\gamma \nu} (1/h)  )^{p_{0}'} \log^{\frac{\delta}{1-p_0}}(1/h)\, dy\\
&\lesssim \log^{\frac{\gamma\nu+\delta}{1-p_0}}(1/h) \to 0
\end{align*}
as $h\to 0$. 
\end{proof}

\section{Subdimensional models}\label{sec:sub_mod}

Recall that in the subdimensional case we set $p_0=d-\frD \leq d$, $\frD = (d-1)\frac{\log 2}{\log\lambda^{-1}}$ if $\lambda,\frD>0$ and $\lambda=0$ if $\frD=0$, for the competitor $u$ there holds 
$$
|\nabla u| \leq |x_d|^{-1}\indicator_\set{2|x_d|\leq d(\bar x,\frC^{d-1}_{\lambda,\gamma})\leq 4 |x_d|},
$$
 and the auxiliary function is defined by
\begin{equation*}
b= \indicator_{ \set{ 
          d(\bar{x},\frC^{d-1}_{\lambda,\gamma}) \leq \frac 12 \abs{x_d}}}\abs{x_d}^{\frD+1-d} \log^{-\gamma\nu} (e+\abs{x_d}^{-1}),
\end{equation*} 
where $\nu = \frD$ if $\frD>0$ and $\nu = d-1$ if $\lambda=\frD=0$.

The function $v_\eta$,  $\eta>0$, is a minimizer of the variational problem \eqref{vtdef}. Recall that $\eta$ here is the boundary value parameter: $v_\eta - \eta u \in W_0^{1,\varphi(\cdot)}(\Omega)$, and by $\Sigma'_\eta$ we denote the set of non-Lebesgue point of $v_\eta$ which lie on $\frS = \frC^{d-1}_{\lambda,\gamma}\times \{0\}$.

In this section we shall use the function $\rho_a \in C^\infty(\setR^d \setminus \frS)$  defined using Lemma~\ref{lem:smooth-indicator} which satisfies
    \begin{itemize}
    \item
      $\indicator_{\set{d(\bar{x},\frC^{d-1}_\lambda) \leq \frac 12 \abs{x_d}}}
      \leq \rho_a \leq \indicator_{\set{d(\bar{x},\frC^{d-1}_\lambda)
          \leq 2 \abs{x_d}}}$.
    \item
      $\abs{\nabla \rho_a} \lesssim \abs{x_d}^{-1}
      \indicator_{\set{\frac 12 \abs{x_d} \leq d(\bar{x},\frC^{d-1}_\lambda) \leq
          2 \abs{x_d}}}$.
    \end{itemize}

\subsection{Piecewise constant variable exponent}

Let $p_{-}\leq p_0 < p_{+}$. Define
$$
      p(x) :=
             \begin{cases}
               p^- &\quad \text{for}\quad \abs{x_d} \leq d(\bar{x},\frC^{d-1}_\lambda),
               \\
               p^+ &\quad \text{for}\quad \abs{x_d} > d(\bar{x},\frC^{d-1}_\lambda).
             \end{cases},
$$
and
\begin{equation}\label{varexp}
             \varphi(x,t):= \frac{t^{p(x)}}{p(x)},\quad \mathcal{F}(w) = \int\limits_\Omega \frac{|\nabla w|^{p(x)}}{p(x)}\, dx.
    \end{equation}

\begin{proposition}\label{prop:px_subd}
Assumption~\ref{ass:harmonic2} is satisfied if $p_{-}<p_0$  or $p_{-}=p_0$ and additionally $\gamma \nu < -1$.
\end{proposition}
\begin{proof}

From Lemmas~\ref{L:u}, \ref{L:b} we get property \eqref{MC1}. It remains to check \eqref{MC2}.

We have
$$
      \mathcal{F}(\eta u) = \eta^{p^-} \mathcal{F}(u),\quad 
      \mathcal{F}^*(s b) = s^{(p^+)'} \mathcal{F}^*(b).
$$
Now, fix $s:= \eta^{p_1-1}$, $p_1 = (p_{-}+p_{+})/2$. Then for suitable large~$t$ (and
  therefore large~$s$) we obtain
  \begin{align*}
    \mathcal{F}(\eta u) +
    \mathcal{F}^*(sb)
    &< \tfrac{1}{p_0} \eta^{p_1} \cdot \kappa +
                \tfrac{1}{p_1'} s^{p_1'} \kappa = \kappa \eta s,
  \end{align*}
  where we have used~$p_{-}\leq p_1$ and $p_+' < p_1'$.  This proves
  Assumption~\ref{ass:harmonic2} for the piecewise constant variable exponent model.
\end{proof}

\begin{theorem}\label{T:varexp_sub}
Let $1<p_{-}<p_{+}<\infty$, $p_{-}\leq d$. Let $\frD = d-p_{-}$ and define $\lambda=2^{(1-d)/\frD}$ if $\frD>0$, $\lambda=0$ if $\frD=0$. Let $\nu=\frD$ if $p_{-}<d$ and $\nu=d-1$ if $p_{-}=d$. For any $\gamma<-1/\nu$ there exists a variable exponent $p(\cdot):\Omega\to [p^-,p^+]$,~$\eta>0$ and a closed subset ~$\Sigma$ of the set  of non-Lebesgue points of the minimizer~$v_\eta$ such that: 
\begin{itemize}
\item $\mu_{\lambda,\gamma} (\Sigma) \geq 1/2$; 
\item $\mathrm{dim}_\mathcal{H} \Sigma =d-p_{-}$;
\item $v_\eta \notin W^{1,s}(\Omega)$ for any $s>p_{-}$.
\end{itemize}
\end{theorem}

\begin{proof}
The proof follows from Proposition~\ref{prop:px_subd} and Theorems~\ref{T:condB}, \ref{T:condA} and Corollary~\ref{cor:regul} by setting $p_0=p_{-}$.
\end{proof}

\subsection{Continuous variable exponent}

Let $\varkappa>0$ and
\begin{equation}\label{sigma_def}
\sigma(t) = \varkappa \frac{\log\log (e^3+t^{-1})}{\log (e+t^{-1})}.
\end{equation}
Let $\rho_a$ be the function defined in the beginning of this section. Let $\xi\in C_0^\infty(\Omega)$, $0\leq \xi \leq 1$, $\xi=1$ in a neighborhood of $\frC^{d-1}_{\lambda,\gamma}\times\{0\}$, $\xi(x)=0$ if $\sigma(|x_d|)>(p_0-1)/10$.   Define the continuous variable exponent model as 
    \begin{alignat*}{2}
      &p^-(x) &&:= p_0-\sigma(\abs{x_d}),
      \\
      &p^+(x) &&:= p_0+\sigma(\abs{x_d}),
      \\
      &p(x) &&:= \xi(x) p^-(x)\,(1-\rho_a)(x)  + \xi(x) p^+(x)\rho_a(x) + (1-\xi(x))p_0.
    \end{alignat*}  
The integrand $\varphi$ and the functional are defined in \eqref{varexp}.

\begin{proposition}\label{prop:cve0}
Let $\varkappa >p_0/2$. Then one can choose the parameter $\gamma$ of the Cantor set $\frC^{d-1}_{\lambda,\gamma}$ so that the continuous variable exponent model satisfies Assumption~\ref{ass:harmonic2}. If $p_0<d$ then one can take any $\gamma$ satisfying $p_0-1-\kappa<\gamma \nu< \varkappa-1$, if $p_0=d-0$ then additionally $\gamma$ must be positive.
\end{proposition}

\begin{proof}
First we verify \eqref{MC1}. We can restrict our attention to the zone where $\xi(x)=1$ and therefore $p(x) = p^-(x)\,(1-\rho_a)(x)  + p^+(x)\rho_a(x)$. Observe that
$$
\{\nabla u\neq 0\} \subset \{x\,:\, p(x)=p_0-\sigma(|x_d|)=d-\frD - \sigma (|x_d|)\}.
$$
Thus we get 
$$
|\nabla u|^{p(x)} = |\nabla u|^{p_0-\sigma(|x_d|)}.
$$
Note that 
\begin{equation}\label{aux1}
t^{\sigma(t)} \lesssim \log^{-\varkappa} (1/t), \quad (\log (1/t))^{\sigma(t)} \to 1 \quad \text{as}\quad t\to 0.
\end{equation}
Therefore,
$$
|\nabla u|^{p(x)}\lesssim \abs{x_d}^{-(d-\frD)} \log^{-\varkappa} (e+\abs{x_d}^{-1})\indicator_\set{2|x_d|\leq d(\bar x,\frC^{d-1}_{\lambda,\gamma})\leq 4 |x_d|}.
$$
Lemma~\ref{L:integr_est} yields $|\nabla u|^{p(\cdot)} \in L^1(\Omega)$ if $\gamma \mu < \varkappa -1$.

Now,
$$
\{b\neq 0\} \subset \{x\,:\, p(x) = p_0 +\sigma (|x_d|)\},
$$
and in this zone we have 
$$
p'(x) = p_0' - \frac{\sigma(|x_d|)}{(p_0-1)(p_0+\sigma-1)}.
$$
Since
$$
\frac{1}{p_0+\sigma-1} = \frac{1}{p_0-1}- \frac{\sigma}{(p_0-1)(p_0-1+\sigma)}
$$
recalling \eqref{aux1} we get 
$$
b^{p'(x)} \lesssim \indicator_{ \set{ 
          d(\bar{x},\frC^{d-1}_{\lambda,\gamma}) \leq \frac 12 \abs{x_d}}} \abs{x_d}^{-(d-\frD)} (\log (e+\abs{x_d}^{-1}))^{-(\kappa+\gamma \nu p_0)(p_0-1)^{-1}}.
$$
By Lemma~\ref{L:integr_est} we have $b^{p'(x)}\in L^1(\Omega)$ provided that 
$$
\gamma \nu(1-p_0') -\varkappa (p_0-1)^{-1} <-1 \Leftrightarrow \gamma \nu >p_0-1- \varkappa.
$$
The interval $ p_0-1 - \varkappa <\gamma \nu < \varkappa -1$ is non-empty provided that $\varkappa >p_0/2$. In the critical case $p_0=d$ this condition also guarantees that we can choose positive $\gamma$. 

Now let us verify \eqref{MC2}. Let $\eta>1$ and $s=\eta^{p_0-1}$. For $\varepsilon>0$ we define the~$\varepsilon$-neighborhood of~$\frS$ by
  \begin{align*}
    \frS_\varepsilon := \set{x \,:\, d(x,\frS) \leq 5 \varepsilon}.
  \end{align*}
  Now, with the definition of~$p^\pm$ we obtain
  \begin{align*}
    \lefteqn{\mathcal{F}(\eta u ) + \mathcal{F}^*(s b)} \quad
    &
    \\
    &=   \int_{\frS_\varepsilon} \phi(x, \eta\abs{\nabla u})\,dx +  \int_{\Omega \setminus \frS_\varepsilon} \phi(x, \eta\abs{\nabla
      u})\,dx 
    \\
    &+ \int_{\frS_\varepsilon} \phi^*(x, sb)\,dx  + \int_{\Omega
      \setminus \frS_\varepsilon} \phi^*(x, sb)\,dx 
    \\
    &\leq 
      \eta^{p_0} \int_{\frS_\varepsilon} \phi(x, \abs{\nabla u})\,dx +  \eta^{p_0 - \sigma(\varepsilon)} \int_{\Omega \setminus
      \frS_\varepsilon} \phi(x, \eta\abs{\nabla u})\,dx 
    \\
    &\quad 
      +  s^{p_0'} \int_{\frS_\varepsilon} \phi^*(x, sb)\,dx 
      + s^{(p_0+\sigma(\varepsilon))'}
      \int_{\Omega \setminus \frS_\varepsilon} \phi^*(x, b)\,dx
    \\
    &\leq 
      \tfrac{1}{p_0}\eta^{p_0}  \bigg(p_0 \int_{\frS_\varepsilon} \phi(x,
      \abs{\nabla u})\,dx +  p_0 \eta^{- \sigma(\varepsilon)} \int_{\Omega \setminus
      \frS_\varepsilon} \phi(x, \eta\abs{\nabla u})\,dx  \bigg)
    \\
    &\quad 
      +  \frac{1}{p_0'} s^{p_0'} \bigg( p_0'\int_{\frS_\varepsilon} \phi^*(x, sb)\,dx 
      + p_0' s^{(p_0+\sigma(\varepsilon))'-p_0'}
      \int_{\Omega \setminus \frS_\varepsilon} \phi^*(x, b)\,dx \bigg)
    \\
    &=:       \tfrac{1}{p_0}\eta^{p_0} \big( \textrm{I} + \textrm{II} \big)
      +       \tfrac{1}{p_0'}s^{p_0'} \big( \textrm{III} + \textrm{IV} \big)
  \end{align*}
  Since
  $\phi(\cdot, \abs{\nabla u}), \phi^*(\cdot, b) \in
  L^1(\Omega)$, we can choose~$\varepsilon_0>0$ small such
  that~$\textrm{I},\textrm{III} < \kappa/2$ for
  all~$\varepsilon \in (0, \varepsilon_0)$. Then choose~$\eta_0$ large
  (depending on~$\varepsilon_0$) such that
  ~$\textrm{II},\textrm{IV} < \kappa/2$ for all~$\eta \geq \eta_0$ and
  $s=\eta^{p_0-1}\geq s_0=\eta^{p_0-1}$. Thus, for all $\eta\geq \eta_0$ we obtain
  \begin{align*}
    \mathcal{F}(\eta u ) + \mathcal{F}^*(sb) &\leq \kappa\Big(
                                                  \tfrac{1}{p_0}\eta^{p_0} +
                                                  \tfrac{1}{p_0'}s^{p_0'} \Big)=\kappa \eta^{p_0}= \kappa \eta s.
  \end{align*} 
  This completes the proof of \eqref{MC2} and thus Assumption~\ref{ass:harmonic2}.
\end{proof}

\begin{theorem}\label{T:varexp2_sub}
Let $1\leq p_0\leq d$ and $\varkappa>p_0/2$. Let $\frD = d-p_{0}$ and define $\lambda=2^{(1-d)/\frD}$ if $\frD>0$, $\lambda=0$ if $\frD=0$. For $\gamma \in (p_0-1-\varkappa, \varkappa-1)$ (for $p_0=d$ additionally $\gamma>0$) there exists a variable exponent $p(\cdot)\in C(\overline\Omega)\cap C^\infty(\overline\Omega \setminus \frS)$, $\frS=\frC^{d-1}_{\lambda,\gamma}\times \{0\}$, with $p=p_0$ on $\frS$ and modulus of continuity \eqref{sigma_def} such that for some $\eta>0$ there exists  a closed subset ~$\Sigma$ of the set  of non-Lebesgue points of the minimizer~$v_\eta$ such that:
\begin{itemize}
\item $\mu^{d-1}_{\lambda,\gamma} (\Sigma)>1/2$
\item $\mathrm{dim}_\mathcal{H} \Sigma =d-p_{0}$;
\item $v_\eta\notin W^{1,s}(\Omega)$ for any $s>p_0$.
\end{itemize}   
\end{theorem}

\begin{proof}
The proof follows from Proposition~\ref{prop:cve0}, Theorems~\ref{T:condB},~\ref{T:condA} and Corollary~\ref{cor:regul}.
\end{proof}

\subsection{Double phase}

Let $\rho_a$ be the function defined in the beginning of this section. Let $\alpha >0$. For $1<p<q<\infty$ we define
$$
      a(x) := \abs{x_d}^\alpha\, \rho_a(x)
$$
and set
\begin{equation}\label{dphdef}
      \varphi(x,t) := \frac{1}{p}t^p + a(x)\frac{1}{q}t^q.\quad       \mathcal{F}(w) := \int\limits_\Omega \bigg(\frac{|\nabla w|^p}{p} + a(x)\frac{|\nabla w|^q}{q} \bigg)\, dx.
\end{equation}

\begin{proposition}\label{double_subd}
Let $q>p_0+\alpha$ and $p< p_0$ or $p=p_0$ and additionally $\gamma \nu <-1$. Then Assumption~\ref{ass:harmonic2}  is satisfied.     
\end{proposition}

\begin{proof}


By construction, 
$$
\{\nabla u \neq 0\} \subset \{x\,:\, \varphi(x,t) = t^p/p\}
$$
Thus 
$$
\varphi(x,|\nabla u|) = \frac{|\nabla u|^p}{p} 
$$
and $\mathcal{F}(u)<\infty$ follows from Lemma~\ref{L:u}.

Now,
$$
\{b\neq 0\} \subset \{a(x) = \abs{x_d}^\alpha\} \subset \{x\,:\, \varphi(x,t) \geq \abs{x_d}^\alpha t^q/q\}.
$$
Therefore 
$$
\varphi^* (x,b) \leq  \abs{x_d}^{-\alpha/(q-1)} \frac{b^{q'}}{q'} \leq  \abs{x_d}^{-q'(d-1-\frD)-\alpha(q'-1)} \log^{-q'\gamma \nu} (e+ \abs{x_d}^{-1}).
$$
 Using  Lemma~\ref{L:integr_est} we get $\mathcal{F}^* (b)<\infty$ if $q>d+\alpha-\frD$ of $q= d+\alpha-\frD$ and $\gamma \nu>d-1+\alpha-\frD$. This proves $\mathcal{F}^* (b)<\infty$ and thus \eqref{MC1}.

Let us prove~\eqref{MC2}. We have 
$$
\varphi(x,\eta\abs{\nabla u}) =\eta^p\varphi(x,\abs{\nabla u}),\quad  \varphi^* (x,sb) \leq s^{q'}\varphi^* (x,b),
$$ 
therefore
$$
      \mathcal{F}(\eta u) = \eta^p \mathcal{F}(u), \quad 
      \mathcal{F}^*(s b) \leq s^{q'-1} \mathcal{F}^*(b).
$$
For $p_1 = (p+q)/2$ we get
  \begin{align*}
    \mathcal{F}(\eta u) +
    \mathcal{F}^*(sb)
    &\leq \frac 1{p_0} \eta^{p_1} \bigg( p_1 \eta^{p-p_1}
      \mathcal{F}(u)\bigg)
      + \frac{1}{p_1} s^{p_1'} \bigg( p_1' s^{q'-p_1'}
      \mathcal{F}^*(b) \bigg).
  \end{align*}
  Now, fix $s:= \eta^{p_1-1}$. Then for suitable large~$\eta$ (and
  therefore large~$s$) we obtain
  \begin{align*}
    \mathcal{F}(\eta u) +
    \mathcal{F}^*(sb)
    &< \tfrac{1}{p_1} \eta^{p_1} \cdot \kappa +
                \tfrac{1}{p_1'} s^{p_1'} \kappa = \kappa \eta s,
  \end{align*}
  where we have used~$p<p_1$ and $q' < p_1'$.  This proves
  Assumption~\ref{ass:harmonic2} for the double phase case.
\end{proof}

\begin{theorem}\label{T:dbl_sub}
Let $1<p\leq d$ and $q>p+\alpha$, $\alpha>0$. Let $\frD = d-p$ and define $\lambda=2^{(1-d)/\frD}$ if $\frD>0$, $\lambda=0$ if $\frD=0$. Let $\nu=\frD$ if $p<d$ and $\nu=d-1$ if $p=d$. Then for any $\gamma<-1/\nu$ there exists a positive weight $a\in C^\alpha(\overline\Omega)$ with $\|a\|_\infty\leq 1$ such that for the double phase functional \eqref{dphdef} for some $\eta>0$ there exists a closed subset ~$\Sigma$ of the set  of non-Lebesgue points of the minimizer~$v_\eta$ such that:\begin{itemize}
\item $\mu^{d-1}_{\lambda,\gamma} (\Sigma)>1/2$,
\item $\mathrm{dim}_\mathcal{H} \Sigma =d-p$;
\item $v_\eta\notin W^{1,s}(\Omega)$ for any $s>p$.
\end{itemize}
\end{theorem}
\begin{proof}
Set $p_0=p$. The proof follows from Proposition~\ref{double_subd},Theorems~\ref{T:condB}, \ref{T:condA} and Corollary~\ref{cor:regul}.
\end{proof}

\subsection{Borderline double phase}

Let $\rho_a$ be the function defined in the beginning of this section. Let $\alpha,\beta \in \mathbb{R}$, $\varkappa>0$, $\varepsilon\in(0,1)$. We define
\begin{equation}\label{def:debl2_0}
      a(x) :=   a_0(x)\, \rho_a(x),\quad a_0(x)=\log^{-\varkappa} (e+|x_d|^{-1})
\end{equation}      
and
\begin{equation}\label{def:dbl2}   
\psi(x,t)= a(x)t^{p_0} \log^{\alpha} (e+t),\quad     \varphi(x,t):=  t^{p_0} \log^{-\beta} (e+t) + \varepsilon^{-1}\psi(x,t).
    \end{equation}
Note that 
\begin{equation}\label{bdd_aux1}
\psi^*(x,t) \lesssim a(x) (t/a(x))^{p_0'} \log^{\alpha/(1-p_0)} (e+ t/a(x))
\end{equation}
for large $t$, and 
\begin{equation}\label{bdd_aux2}
\varphi^*(x,t) \leq(\varepsilon^{-1}\psi(x,\cdot))^*(t) = \varepsilon^{-1}\psi^*(x,\varepsilon t).
\end{equation}

\begin{proposition}
Let $\alpha+\beta>p_0+\varkappa$. For $p_0=d-0$ we additionally require that $\beta>1$. Then $\psi^*(x,b)\in L^1(\Omega)$ and \eqref{MC1} holds if $\varkappa -\alpha +p_0-1 < \gamma \nu <\beta-1$ (for $p_0=d-0$ we additionally require $\gamma<0$).    
\end{proposition}

\begin{proof}

We have
$$
\{\nabla u \neq 0\} \subset \{x\,:\, a(x)=0\}.
$$
Therefore,
\begin{align*}
\varphi(x,\abs{\nabla u}) =  \abs{\nabla u}^{p_0} \log^{-\beta} (e+\abs{\nabla u})\\
\leq\abs{x_d}^{-p_0} \log^{-\beta} (e+\abs{x_d})\indicator_\set{2|x_d|\leq d(\bar x,\frC^{d-1}_{\lambda,\gamma})\leq 4 |x_d|}.
\end{align*}
By Lemma~\ref{L:integr_est}, $\varphi(x,|\nabla u|)\in L^1(\Omega)$ if $\gamma \nu-\beta<-1$. 

Now,
$$
\{b\neq 0\} \subset \{x\,:\, a(x)=a_0(x)\} . 
$$
Since 
$$
\varphi^*(x,b) \leq \varepsilon^{-1} \psi^*(x,\varepsilon b) \leq C(\varepsilon) \psi^*(x,b)
$$
it remains to evaluate
\begin{align*}
\psi^*(x,b) \lesssim a_0(x) (b/a_0(x))^{p_0'} \log^{\alpha/(1-p_0)} (e+ b/a_0(x))\\
\lesssim |x_d|^{-p_0} (\log(e+|x_d|^{-1}))^{(\varkappa-\gamma \nu p_0- \alpha)/(p_0-1)}  \indicator_\set{\mathrm{dist}\, (\bar x,\frC^{d-1}_{\lambda,\gamma}\leq |x_d|/2)}.
\end{align*}
By Lemma~\ref{L:integr_est} we get the integrability of $\psi^*(x,b)$, and therefore $\varphi^*(x,b)$, if 
$$
\varkappa -\gamma \nu-\alpha < 1-p_0\Leftrightarrow   \gamma \nu>\varkappa -\alpha+p_0-1.
$$
The interval 
$$
\varkappa -\alpha +p_0-1 < \gamma \nu <\beta-1
$$
is non-empty provided that $\alpha +\beta > p_0+\varkappa$. In the critical case $p_0=d-0$ this interval has a non-empty intersection with the positive semiaxis if $\beta>1$.
\end{proof}

\begin{proposition}\label{border:subd}
For any $\kappa>0$ there exists $\varepsilon>0$ such that Assumption~\ref{ass:harmonic3} holds with $\kappa$.
\end{proposition}
\begin{proof}
Let $\eta=1$, $s=\varepsilon^{-1}\sigma$ and denote $c_1=\mathcal{F}(u)$. We get 
$$
\mathcal{F}(\eta u) + \mathcal{F}^*(sb) \leq  c_1 + \varepsilon^{-1}\sigma \int\limits_\Omega \psi^*(x,b\sigma)\, dx.
$$
By the Lebesgue dominated convergence theorem $\sigma^{-1}\int\limits_\Omega \psi^*(x,b\sigma)\, dx\to 0$ as $\sigma\to 0$. Choose sufficiently small $\sigma>0$ such that  $\sigma^{-1}\int\limits_\Omega \psi^*(x,b\sigma)\, dx < \kappa /2$ and then $\varepsilon>0$ such that $c_1 \varepsilon<\kappa \sigma/2$. 
\end{proof}

\begin{theorem}\label{T:bdd_sub}
Let $1<p_0\leq d$, $\alpha +\beta > \varkappa +p_0$, and additionally $\beta>1$ if $p_0=d-0$. Let $\frD = d-p_0$ and define $\lambda=2^{(1-d)/\frD}$ if $\frD>0$, $\lambda=0$ if $\frD=0$. Let $\nu=\frD$ if $p<d$ and $\nu=d-1$ if $p=d$. For any $\gamma\in ((\varkappa-\alpha+p_0-1)/\nu, (\beta-1)/\nu )$ (if $p_0=d-0$ additionally $\gamma$ must be positive) there exists a weight $a(x)$ with the modulus of continuity $\omega(\rho)=C \log^{-\kappa}(e+\rho^{-1})$ and $\varepsilon>0$ such that for the borderline double phase model \eqref{def:dbl2} there exists  a closed subset ~$\Sigma$ of the set  of non-Lebesgue points of the minimizer~$v_\eta$ such that:
\begin{itemize}
\item $\mu^{d-1}_{\lambda,\gamma} (\Sigma)>1/2$. 
\item $\mathrm{dim}_\mathcal{H} \Sigma =d-p_0$;
\item $v_1 \notin W^{1,s}(\Omega)$ for any $s>p_0$.
\end{itemize}
\end{theorem}
 \begin{proof}
Set $p_0=p$. The proof follows from Proposition~\ref{border:subd},Theorems~\ref{T:condB}, \ref{T:condA} and Corollary~\ref{cor:regul}.

\end{proof}

\section{Irregularity of minimizers: superdimensional case}\label{sec:super}

In this section we show that under Assumption~\ref{ass:harmonic2} (or Assumption~\ref{ass:harmonic3}) the regularity of the $W$-minimizer $v_\eta$ is roughly speaking no better than the regularity of the competitor function $u$. Recall that by $u$ and $b$ we denote the competitor and auxiliary functions, respectively, from Definition~\ref{def:fractal-examples}, built using the Cantor set $\frC_{\lambda,\gamma}$ and the contact set $\frS=\{0\}^{d-1}\times \frC_{\lambda,\gamma}$. Here $p_0=\frac{d-\frD}{1-\frD}\geq d$, $\lambda=2^{-1/\frD}$ if $\frD>0$ and $\lambda=0$ if $\frD=0$. The minimizer $v_\eta$ is defined in \eqref{vtdef}. Recall the definition of the auxiliary function 
$$
b= \abs{\bar x}^{1-d} \indicator_\set{2\abs{\bar x} \leq d (x_d, \frC_{\lambda,\gamma})\leq 4 \abs{\bar x}}.
$$

Let $m\in \mathbb{N}$. By $\xi^{\pm}_{m,j}$, $j=1,\ldots,2^m$ we denote the left and right end points of the intervals from the pre-Cantor set $\frC_{\lambda,\gamma}^{(m)}$ of $m$-th generation. Let $\xi^+_{m,0}=-1$, $\xi^-_{m,2^m+1}=1$. Define the sets  
\begin{align*}
&\mathcal{N}^+_{m,j} =\mathcal{N}^-_{m,j+1} \\
&:=\{(\bar x,x_d)\,:\, \xi^{+}_{m,j}< x_d<\xi^{-}_{m,j+1},\   \frac{1}{4}\mathrm{dist}(x_d, \frC^{(m)}_{\lambda,\gamma})\leq |\bar x|\leq   \frac{1}{2}\mathrm{dist}(x_d, \frC^{(m)}_{\lambda,\gamma}) \}  .
\end{align*}

Note that for $\xi^{+}_{m,j}< x_d<\xi^{-}_{m,j+1}$ there holds $\mathrm{dist}(x_d, \frC^{(m)}_{\lambda,\gamma}) = \mathrm{dist}(x_d, \frC_{\lambda,\gamma})$
 
Define the restricted Riesz potentials as in Section~\ref{ssec:rrp} and denote
$$
I(f) (\bar{0},\xi^{\pm}_{m,j}) = I_1^{\pm}(f\indicator_{\mathcal{N}^{\pm}_{m,j}})(\bar{0},\xi^{\pm}_{m,j}).
$$
Then 
$$
\sum_{j=1}^{2^m} (I(f)(\bar{0},\xi^{-}_{m,j})+I(f)(\bar{0},\xi^{+}_{m,j})) \leq \int\limits_\Omega fb \, dx.
$$
Assume that $v\in W^{1,1}(\Omega)$ and $|\nabla v|b\in L^1(\Omega)$. Let $\tau= 1/100$. For $x=(\bar{0}, \xi^{\pm}_{m,j})$ we consider a sequence of balls $B_k$ in $\mathcal{N}^{\pm}_{m,j}$ with centers $z_k\to x$ and radii $\tau |(z_k)_d - \xi^{\pm}_{m,j}|$. From the estimates of Section~\ref{ssec:trace} we get the convergence of the sequence of averages $\langle v \rangle_{B_k}$ to some limit which we shall further denote as simply $v(\bar{0}, \xi^{\pm}_{m,j})$. Moreover, from the same estimates it follows that 
$$
|v(\bar{0}, \xi^{-}_{m,j+1}) - v(\bar{0}, \xi^{+}_{m,j})|\leq C(d)(I(f)(\bar{0},\xi^{-}_{m,j})+I(f)(\bar{0},\xi^{+}_{m,j})),
$$
and so 
\begin{equation}\label{superdM}
\sum_{j=0}^{2^m}|v(\bar{0}, \xi^{-}_{m,j+1}) - v(\bar{0}, \xi^{+}_{m,j})|\leq C(d) \int\limits_\Omega |\nabla v|b\,dx.
\end{equation}

\begin{proposition}\label{prop:sd1}
Under Assumption~\ref{ass:harmonic2} (or Assumption~\ref{ass:harmonic3} with sufficiently small $\kappa$) for the minimizer $v_\eta$ we have 
\begin{equation}\label{sd1}
S(\eta):=\sum_{j=1}^{2^m} (v_\eta(\bar{0},\xi^{+}_{m,j}) - v_\eta(\bar{0},\xi^{-}_{m,j})) \geq \frac{\eta}{2}.
\end{equation}
\end{proposition}
\begin{proof}
We extend $v_\eta$ by $\eta/2$ to $(-1,1)^{d-1}\times (1,\infty)$ and by $-\eta/2$ to $(-1,1)^{d-1}\times (-\infty,-1)$. There holds 
\begin{align*}
\sum_{j=0}^{2^m}(v_\eta(\bar{0}, \xi^{-}_{m,j+1}) - v_\eta(\bar{0}, \xi^{+}_{m,j}))+S(t) \\
=v(\bar 0,1)-v(\bar 0,-1)=(\eta/2)-(-\eta/2)=\eta.
\end{align*}
Using estimate \eqref{superdM} we get
$$
S(\eta) = \eta - C(d)\int\limits_\Omega |\nabla v_\eta|b\,dx.
$$
Now using Assumption~\ref{ass:harmonic2} (Assumption~\ref{ass:harmonic3} with sufficiently small $\kappa=\kappa(d)$) by Lemma~\ref{lem:riesz-sublinear00} (Lemma~\ref{lem:riesz-sublinear01}) we get 
$$
C(d)\int\limits_\Omega |\nabla v_\eta|b\,dx < \frac{\eta}{2}.
$$
Thus $S(\eta) \geq \eta/2$ as required.
\end{proof}


Now we are ready to prove Theorem~\ref{T:condC}.

\begin{proof}[Proof of Theorem~\ref{T:condC}.]
First consider the case $p_0>d$. Then the length of the pre-Cantor interval of the $m$-th generation is $l_m = \lambda^m m^\gamma$. For large $m$ we have
$$
m\leq c+ \frac{\log (1/l_m)}{\log \lambda} + \gamma \frac{\log \log (1/l_m )}{\log \lambda},
$$
so recalling that $\lambda^{-\frD}=2$ we get 
$$
2^m \leq C l_m^{-\frD} (\log (1/l_m))^{\gamma \frD}.
$$
 From \eqref{sd1} we have
\begin{equation}\label{contr}
2^{-m}\sum_{j=1}^{2^m} \frac{|v(\bar{0},\xi^{+}_{m,j}) - v(\bar{0},\xi^{-}_{m,j})|}{ l_m^{\frD} (\log (1/l_m))^{-\gamma \nu}} \geq \frac{\eta}{2C}, 
\end{equation}
where $\nu=\frD$. Let $\omega(\rho)=\rho^\frD \log^{-\gamma\frD}(e+\rho^{-1})$. Any modulus of continuity $\omega_1(\rho)$ such that $\omega_1(\rho)/\omega(\rho \to 0 $ as $\rho \to \infty$ gives contradiction.

For $p_0=d-0$  the length of the pre-Cantor interval of the $m$-th generation is $l_m = \exp(-2^{m/\gamma})$, thus $m=\gamma (\log 2)^{-1} \log\log(1/l_m)$ and $2^m = (\log (1/l_m))^\gamma$ and we get \eqref{contr} with $\frD=0$ and $\nu=1$. Any modulus of continuity better than $\omega(\rho)=\log^{-\gamma}(e+\rho^{-1})$ contradicts \eqref{contr}.
\end{proof}

Now we shall obtain the absence of the higher gradient integrability. Let $\tau=(1-2\lambda)/4$. Denote 
$$
\mathcal{M}_{m,j} = B^{d-1}_{\eta l_m}(\bar 0)\times (\xi^{-}_{m,j}-\eta l_m,\xi^{+}_{m,j}+\eta l_m).
$$
For any $j=1,\ldots,2^m$  choose in $\mathcal{N}^{\pm}_{m,j}$ a couple of balls $B^{\pm}_{m,j}$  with centers $z^{\pm}_{m,j}$ such that $|(z^{\pm}_{m,j})_d - \xi^{\pm}_{m,j}|< \tau l_m /100$ and radii $|(z^{\pm}_{m,j})_d - \xi^{\pm}_{m,j}|/100$. From the estimates of Section~\ref{ssec:trace} it follows that 
\begin{equation}\label{e1}
\sum_{j=1}^{2^m} (|\langle v \rangle_{B^+_{m,j}} -  v(\bar 0,\xi^{+}_{m,j})|+|\langle v \rangle_{B^-_{m,j}} -  v(\bar 0,\xi^{-}_{m,j})| \leq C(d,\lambda) \int\limits_\Omega |\nabla v|b\, dx.
\end{equation}
Here $|\langle v \rangle_{B^{\pm}_{m,j}} = |B^{\pm}_{m,j}|^{-1}\int_{B^{\pm}_{m,j}} v$.

On the other hand,
\begin{equation}\label{e2}
||\langle v \rangle_{B^+_{m,j}}- \langle v \rangle_{B^-_{m,j}}|  \leq C(d,\lambda) l_m^{1-d}\int\limits_{\mathcal{M}_{m,j}} |\nabla v|\, dx.
\end{equation}
Combining \eqref{e1} and \eqref{e2} we obtain the following inequality:
\begin{equation}\label{e3}
\begin{gathered}
\sum_{j=1}^{2_m} |v(\bar 0,\xi^{+}_{m,j})- v(\bar 0,\xi^{-}_{m,j})| \\
\leq  C(d,\lambda) \int\limits_\Omega |\nabla v|b\, dx +C(d,\lambda) \sum_{j=1}^{2^m}l_m^{1-d}\int\limits_{\mathcal{M}_{m,j}} |\nabla v|\, dx.
\end{gathered}
\end{equation}

\begin{theorem}\label{T:abssuper}
Let $\Psi$ be a generalized Orlicz function and
\begin{equation}\label{ass:abssup}
\sum_{j=1}^{2^m} \int\limits_{\mathcal{M}_{m,j}} \Psi^*(x,l_m^{1-d})\, dx\to 0   
\end{equation}
as $m\to \infty$. Then under Assumption~\ref{ass:harmonic2} (or Assumption~\ref{ass:harmonic3} with sufficiently small $\kappa$), we have
$$
\int\limits_\Omega \Psi(x,|\nabla v_\eta|)\, dx = +\infty.
$$
\end{theorem}

\begin{proof}
From Proposition~\ref{prop:sd1} (estimate \eqref{sd1}) and inequality we get 
$$
\int\limits_\Omega |\nabla v|b\, dx +\sum_{j=1}^{2^m}l_m^{1-d}\int\limits_{\mathcal{M}_{m,j}} |\nabla v_\eta|\, dx \geq c\eta, \quad c>0.
$$
By Lemma~\ref{lem:riesz-sublinear00} (or Lemma~\ref{lem:riesz-sublinear01}), the first term on the left-hand side can be made smaller than $c\eta/2$. Thus 
$$
\sum_{j=1}^{2^m}l_m^{1-d}\int\limits_{\mathcal{M}_{m,j}} |\nabla v|\, dx \geq c\eta/2.
$$
The proof is completed as in Theorem~\ref{T:absenceHI0}.
\end{proof}

\begin{corollary}\label{corr:integ1}
Under Assumption~\ref{ass:harmonic2} (or Assumption~\ref{ass:harmonic3} with sufficiently small $\kappa$) 
$$
 \int\limits_{\Omega}\abs{\nabla v_\eta}^{p_0}\log^{\delta}(e+\abs{\nabla v_\eta} ) dx=\infty
$$
for any~$\delta>(p_0-1)\gamma\nu$.
In particular the minimizer $v_\eta \notin W^{1,s}(\Omega)$ for any $s>p_0$.
\end{corollary}

\begin{proof}

Let  $\Psi(\tau)=\tau^{p_0}\log^\delta(e+\tau)$ in Theorem~\ref{T:abssuper}. We have 
\begin{align*}
 \Psi^*(\tau)\lesssim \tau^{p_{0}'} \log^{\frac{\delta}{1-p_0}}(e+\tau). 
\end{align*}

 If $\lambda>0$ then since~$l_m=\lambda^m m^{\gamma}$ and~$(1-d)p_0'=-(d-\frD)$,~$\lambda^{-\frD}=2$ we have
\begin{align*}
\sum_{j=1}^{2^m} \int\limits_{\mathcal{M}_{m,j}} \Psi^*(x,l_m^{1-d})\, dx \\ 
\lesssim 2^m(l_m)^d  (l_m^{1-d})^{p_0'}\log^{\frac{\delta}{1-p_0}}(1/l_m)  \lesssim m^{\gamma\frD+\frac{\delta}{1-p_0}} \to 0 
\end{align*}
as $m\to \infty$. If $\lambda=0$ then recalling that $l_m=\exp(-2^{m/\gamma})$ we easily see the required property.
\end{proof}

\section{Superdimensional models} \label{sec:super_mod}

Recall that in this case we set $p_0=\frac{d-\frD}{1-\frD}$, $\lambda = 2^{-1/\frD}$, $\nu=\frD$ if $\frD>0$ and $\nu=1$ if $\frD=0$. The competitor and the auxiliary function from Definition~\ref{def:fractal-examples} satisfy
\begin{align*}
 \abs{\nabla u} \lesssim \abs{\bar{x}}^{\frD-1} \log^{-\gamma \nu}(e+ |\bar x|^{-1}) \indicator_{ \set{ 
                     d(x_d,\frC_{\lambda,\gamma}) \leq \frac 12 \abs{\bar{x}}}},\\
             b= \abs{\bar x}^{1-d} \indicator_\set{2\abs{\bar x} \leq d (x_d, \frC_{\lambda,\gamma})\leq 4 \abs{\bar x}}.        
\end{align*}
Given an integrand $\varphi$, the minimizer $v_\eta$ is defined in \eqref{vtdef}, $v_\eta-tu \in W_0^{1,\varphi(\cdot)}(\Omega)$.

Let $\rho_a \in C^\infty(\setR^d \setminus \frS)$ be such that (using     Lemma~\ref{lem:smooth-indicator})
    \begin{itemize}
    \item
      $\indicator_{\set{d(x_d,\frC_\lambda) \leq \frac 1 2 \abs{\bar{x}}}}
      \leq \rho_a \leq \indicator_{\set{d(x_d,\frC_\lambda)
          \leq 2 \abs{\bar{x}}}}$.
    \item
      $\abs{\nabla \rho_a} \lesssim \abs{\bar{x}}^{-1}
      \indicator_{\set{\frac 12 \abs{\bar{x}} \leq
          d(x_d,\frC_\lambda) \leq 2 \abs{\bar{x}}}}
      $.
    \end{itemize}

\subsection{Piecewise constant variable exponent}

Let $p_0>d$ or $p_0=d-0$ and $p_{-}\leq p_0<p_{+}$. We set 
$$
p(x) = \begin{cases}
p_{-}\quad \text{for}\quad \mathrm{dist}\, (x_d,\frC_{\lambda,\gamma})\leq |\bar x|,\\
p_{+}\quad \text{for}\quad \mathrm{dist}\, (x_d,\frC_{\lambda,\gamma})> |\bar x|.
\end{cases}
$$
and define the integrand and functional as in \eqref{varexp}.

\begin{proposition}\label{prop:px_superd}
Assumption~\ref{ass:harmonic2} is satisfied if $p_{-}<p_0$ or $p_{-}=p_0$ and $\gamma \nu>1/(p_0-1)$.
\end{proposition}

\begin{proof}
The property \eqref{MC1} follows from Lemmas~\ref{L:u}, \ref{L:b}. Property \eqref{MC2} is obtained in exactly the same manner as in Proposition~\ref{prop:px_subd}. 
\end{proof}

\begin{theorem}\label{T:varexp_super}
Let $1<p_{-}<p_{+}<\infty$, $p_{-}\geq d$. Let $\frD =\frac{p_{-}-d}{p_{-}-1}$, $\lambda=2^{-1/\frD}$ if $p_{-}>d$ and $\lambda=0$ if $p_{-}=d$, $\nu=\frD$ if $p_{-}>d$ and $\nu=1$ if $p_{-}=d$. For any $\gamma>((p_{-}-1)\nu)^{-1}$ there exists a variable exponent $p(\cdot):\Omega\to \mathbb{R}$ such that $p_{-}\leq p(\cdot)\leq p_{+}$ in $\Omega$, and $\eta>0$ such that: 
\begin{itemize}
\item The modulus of continuity of $v_\eta$ is not better than  $C\rho^\frD \log^{-\gamma \nu} (e+\rho^{-1})$;
\item $v_\eta \notin W^{1,s}(\Omega)$ for any $s>p_{-}$.
\end{itemize}
\end{theorem}

\begin{proof}
The proof follows from Proposition~\ref{prop:px_superd} and Theorems \ref{T:condA},~\ref{T:condC}, by setting $p_0=p_{-}$.
\end{proof}

\subsection{Continuous variable exponent}

Let the function $\sigma$ be defined by \eqref{sigma_def}.  Let $\xi\in C_0^\infty(\Omega)$, $0\leq \xi \leq 1$, $\xi=1$ in a neighborhood of $\{0\}^{d-1}\times \frC_{\lambda,\gamma}$, $\xi(x)=0$ if $\sigma(|\bar x|)>(p_0-1)/10$.   Define  
    \begin{alignat*}{2}
      &p^-(x)& &:= p_0-\sigma(\abs{\bar{x}}),
      \\
      &p^+(x)& &:= p_0+\sigma(\abs{\bar{x}}),
      \\
      &p(x)& &:= \xi(x) p^-(x)\,\rho_a(x)  +\xi(x) p^+(x)(1-\rho_a)(x) + (1-\xi(x)) p_0.
    \end{alignat*}  
The integrand and the functional are defined by \eqref{varexp}.

\begin{proposition}\label{prop:cve1}
Let $p_0>d$ and $\kappa >p_0/2$ or $p_0=d+0$ and $\varkappa > d-1$. Let $\frD=\frac{p_0-d}{p_0-1}$, $\lambda=2^{-1/\frD}$ if $p_0>d$ and $\lambda=0$ if $p_0=d$, $\nu=\frD$ if $\frD>0$ and $\nu=1$ if $\frD=0$. Then one can choose the parameter $\gamma$ of the Cantor set $\frC_{\lambda,\gamma}$ so that the continuous variable exponent model satisfies Assumption~\ref{ass:harmonic2}. If $p_0>d$ then one can take any $\gamma$ satisfying $(1-(1-\frD)\varkappa)(p_0-1)^{-1}<\gamma \nu <-1+\varkappa (d-1)(p_0-1)^{-2}$, if $p_0=d+0$ then additionally $\gamma$ must be positive.
\end{proposition}

\begin{proof}

First we verify \eqref{MC1}. It is sufficient to check the neighborhood of $\{0\}^{d-1}\times \frC_{\lambda,\gamma}$ where $\xi =1$ and $p(x)=p^-(x)\,\rho_a(x)  + p^+(x)(1-\rho_a)(x)$. We have 
$$
\{\nabla u \neq 0\}\subset \{p(x)=p_0-\sigma(|\bar x|)\}
$$
Therefore,
$$
|\nabla u|^{p(x)} = |\nabla u|^{p_0-\sigma(|\bar x|)}.
$$
Recalling \eqref{aux1} we get
$$
|\nabla u|^{p(x)} \lesssim \abs{\bar x}^{-(d-\frD)} (\log (e+|\bar x|^{-1}))^{-\gamma \nu p_0 - (1-\frD)\varkappa}\indicator_{ \set{ d(x_d,\frC_{\lambda,\gamma}) \leq \frac 12 \abs{\bar{x}}}}.
$$
We apply Lemma~\ref{L:integr_est} with the parameters $\beta =  d- \frD$ and $\delta = -\gamma \mu p_0 - (1-\frD)\varkappa$ to get $|\nabla u|^{p(\cdot)} \in L^1(\Omega)$ if 
$$
\gamma \nu(1-p_0) - (1-\frD)\varkappa <  -1\Leftrightarrow \gamma \nu > \frac{1-(1-\frD)\varkappa}{p_0-1}.
$$

Now,
$$
\{b\neq 0\} \subset \{p(x) = p_0 +\sigma (|\bar x|)\}
$$
and in this zone we have
$$
p'(x) = p_0' - \frac{\sigma(|\bar x|)}{(p_0-1)(p_0+\sigma(|\bar x|)-1)}.
$$
Recalling \eqref{aux1} and using that
$$
\frac{1}{p_0+\sigma-1} = \frac{1}{p_0-1}- \frac{\sigma}{(p_0-1)(p_0-1+\sigma)}.
$$
thus
$$
b^{p'(x)} \lesssim \abs{\bar x}^{-(d-\frD)} (\log (e+\abs{\bar x}^{-1}))^{-\varkappa(d-1)(p_0-1)^{-2}} \indicator_\set{2\abs{\bar x} \leq d (x_d, \frC_{\lambda,\gamma})\leq 4 \abs{\bar x}} 
$$
By Lemma~\ref{L:integr_est} we have $b^{p'(x)}\in L^1(\Omega)$ provided that 
$$
\gamma \nu -\varkappa (d-1) (p_0-1)^{-2} <-1 \Leftrightarrow \gamma \nu <-1+ \varkappa (d-1) (p_0-1)^{-2}.
$$
The interval 
$$
\frac{1-(1-\frD)\varkappa}{p_0-1}<\gamma \nu <-1+\varkappa \frac{d-1}{(p_0-1)^2}
$$
is nonempty provided that $\varkappa >p_0/2$. In the critical case $p_0=d+0$ we also require that $\varkappa>d-1$ to choose positive $\gamma$. The proof of \eqref{MC2} then proceeds exactly as in Proposition~\ref{prop:cve0}.
\end{proof}

\begin{theorem}\label{T:varexp2_super}
Let $p_0\geq d$ and $\varkappa>p_0/2$. Let $\frD=\frac{p_0-d}{p_0-1}$, $\lambda=2^{-1/\frD}$ if $p_{0}>d$ and $\lambda=0$ if $p_0=d$, $\nu=\frD$ if $p_0>d$ and $\nu=1$ if $p_0=d$. For any $\gamma\in(1-(1-\frD)\varkappa)(p_0-1)^{-1}<\gamma \nu <-1+\varkappa (d-1)(p_0-1)^{-2}$ ( if $p_0=d$ then we additionally require $\gamma>0$) there exists a variable exponent $p(\cdot)\in C(\overline\Omega)$ with $p=p_0$ on $\frS=\{0\}^{d-1}\times\frC_{\lambda,\gamma}$ and modulus of continuity \eqref{sigma_def} such that for some $\eta>0$ there holds 
\begin{itemize}
\item The modulus of continuity of $v_\eta$ is not better than $C\rho^\frD \log^{-\gamma \nu} (e+\rho^{-1})$;
\item $v_\eta\notin W^{1,s}(\Omega)$ for any $s>p_0$.
\end{itemize}   
\end{theorem}

\begin{proof}
The proof follows from Proposition~\ref{prop:cve1} and Theorems~\ref{T:condA}, ~\ref{T:condC}.
\end{proof}

\subsection{Double phase}

Let $\alpha>0$, and define
    $$
      a(x) := \abs{\bar x}^\alpha\, (1-\rho_a)(x).
$$
Then for $1<p<q<\infty$ the integrand $\varphi$ and functional $\mathcal{F}$ are then defined as in \eqref{dphdef}.

\begin{proposition}\label{prop:dbl_superd}
Let $q>p_0+\alpha\frac{p_0-1}{d-1}$ and $p<p_0$ or $p=p_0$ and $\gamma \nu (p_0-1)>1$. Then Assumption~\ref{ass:harmonic2} is satisfied.
\end{proposition}

\begin{proof}
 
By construction 
$$
\{\nabla u \neq 0\} \subset \{a(x)=0\}.
$$ 
Thus $\varphi(x,\nabla u) = |\nabla u|^{p}$ and the property $\mathcal{F}(u)<\infty$ follows from Lemma~\ref{L:u}.

Now, 
$$
\{b\neq 0\} \subset \{x\,:\, a(x) = |\bar x|^\alpha\}
$$
and thus 
\begin{align*}
\varphi^* (x,b) \leq \frac{1}{q'} |\bar x|^\frac{\alpha}{1-q} |b|^{q'}\\
\leq \frac{1}{q'} |\bar x|^{-q'(1-d) + \alpha(1-q)^{-1}}\indicator_\set{2\abs{\bar x} \leq d (x_d, \frC_{\lambda,\gamma})\leq 4 \abs{\bar x}}.
\end{align*}
 By Lemma~\ref{L:integr_est} we get $\varphi^*(x,b)\in L^1(\Omega)$ if 
$$
q'(d-1) + \frac{\alpha}{q-1}< d-\frD \Leftrightarrow q>\frac{d-\frD+\alpha}{1-\frD},
$$
 which is the assumption $q>p_0 + \alpha \frac{p_0-1}{d-1}$. This proves $\mathcal{F}^*(b)<\infty$ and thus \eqref{MC1}. The property \eqref{MC2} is proved exactly as in Proposition~\ref{double_subd}.
\end{proof}

\begin{theorem}\label{T:dbl_super}
Let $p\geq d$ and $q>p+\alpha\frac{p-1}{d-1}$, $\alpha>0$. Let $\frD = \frac{p-d}{d-1}$, $\nu=\frD$ if $p>d$ and $\nu=1$ if $p=d$.  Then for any $\gamma>(\nu(p-1))^{-1}$ there exists a positive weight $a\in C^\alpha(\overline\Omega)$ such that for the double phase functional \eqref{dphdef} for some $\eta>0$ there holds
\begin{itemize}
\item The modulus of continuity of $v_\eta$ is not better than  $C\rho^\frD \log^{-\gamma \nu} (e+\rho^{-1})$;
\item $v_\eta\notin W^{1,s}(\Omega)$ for any $s>p$.
\end{itemize}
\end{theorem}
\begin{proof}
Set $p_0=p$. The proof follows from Proposition~\ref{prop:dbl_superd} and Theorems~\ref{T:condA}, \ref{T:condC}.
\end{proof}

\subsection{Borderline double phase}

Let $\alpha,\beta \in \mathbb{R}$ and $\varkappa>0$. We define
$$
      a(x) :=   a_0(x)\, (1-\rho_a(x)), \quad a_0(x)=\log^{-\varkappa} (e+\abs{\bar x}^{-1})
$$
and then define the integrand $\varphi(x,t)$ as in \eqref{def:dbl2}:
$$      
\psi(x,t)=a(x)t^{p_0} \log^{\alpha} (e+t),\quad \varphi(x,t):= t^{p_0} \log^{-\beta} (e+t) + \varepsilon^{-1} \psi(x,t).
$$
Here $\varepsilon\in(0,1)$ is the number to be defined later.

\begin{proposition}
Let $\alpha+\beta>p_0+\varkappa$. For $p_0=d+0$ we additionally require $\alpha>p_0+\varkappa-1$. Then $\psi^*(x,b)\in L^1(\Omega)$ and \eqref{MC1} holds if $1-\beta<\gamma\nu(p_0-1)<\alpha+1-p_0-\varkappa$ (for $p_0=d+0$ we additionally require $\gamma>0$).
\end{proposition}

\begin{proposition}
We have 
$$
\{\nabla u\neq 0\} \subset \{\varphi(x,t) = \varepsilon \eta^{p_0}\log^{-\beta} (e+t)\}.
$$
Thus 
$$
\varphi(x,|\nabla u|) \lesssim  \abs{\bar{x}}^{-(d-\frD)} \log^{-\gamma \nu p_0 - \beta} (e+|\bar x|^{-1}) \indicator_{ \set{d(x_d,\frC_{\lambda,\gamma}) \leq \frac 12 \abs{\bar{x}}}}.
$$
Using~\ref{L:integr_est} we get $\varphi(x,|\nabla u|) \in L^1(\Omega)$ provided that $\gamma \nu(1-p_0)-\beta<-1$. 

Now, 
$$
\{b\neq 0\} \subset \{a(x) = \log^{-\varkappa} (e+\abs{\bar x}^{-1})\}
$$
and so recalling \eqref{bdd_aux1}, \eqref{bdd_aux2} we get
\begin{align*}
\varphi^*(x,b) \leq \varepsilon^{-1} \psi^*(x,\varepsilon b) \leq C(\varepsilon) \psi^*(x,b)\\
 \lesssim a(x) (b/a(x))^{p_0'} \log^{\alpha/(1-p_0)} (e+ b/a(x))\\
\lesssim  \abs{\bar x}^{-(d-\frD)}\log^\frac{\varkappa-\alpha}{p_0-1} (e+\abs{\bar x}^{-1})        \indicator_\set{2\abs{\bar x} \leq d (x_d, \frC_{\lambda,\gamma})\leq 4 \abs{\bar x}}.
\end{align*}
By Lemma~\ref{L:integr_est} we get $\varphi^*(x,b),\psi^*(x,b)\in L^1(\Omega)$ provided that $\gamma \nu (p_0-1)+ \varkappa-\alpha<1-p_0$. The interval 
$$
1-\beta<\gamma \nu(p_0-1)<\alpha+1-p_0-\varkappa
$$
is non-empty provided that $\alpha+\beta > p_0+\varkappa$. If $p_0=d+0$ we additionally require that $\alpha>p_0+\varkappa+1$ for $\gamma$ to be positive.
\end{proposition}

\begin{proposition}
For any $\kappa>0$ there exists $\varepsilon>0$ such that Assumption~\ref{ass:harmonic3} holds with $\kappa$.
\end{proposition}

\begin{proof}
The proof repeats the proof in the subdimensional case (Proposition~\ref{border:subd}).
\end{proof}

\begin{theorem}\label{T:bdd_super}
Let $p_0\geq d$,$\alpha+\beta>p_0+\varkappa$, $\frD = \frac{p_0-d}{d-1}$, $\lambda=2^{-1/\frD}$ and $\nu=\frD$ if $p_0>d$, $\lambda=0$ and $\nu=1$ if $p_0=d$. If $p_0=d$ we additionally require $\alpha+1>p_0+\varkappa$ For any $\gamma$ satisfying $1-\beta<\gamma\nu(p_0-1)<\alpha+1-p_0-\varkappa$ (for $p_0=d+0$ we additionally require $\gamma>0$) there exists a weight $a$ with the modulus of continuity $C (\log(e+\rho)^{-1})^{-\varkappa}$ and $\varepsilon>0$ such that for the integrand \eqref{def:dbl2}:
\begin{itemize}
\item The modulus of continuity of $v_1$ is not better than  $C\rho^\frD \log^{-\gamma \nu} (e+\rho^{-1})$;
\item $v_1\notin W^{1,s}(\Omega)$ for any $s>p$.
\end{itemize}
\end{theorem}

\section{Comparison with the approach of \cite{BalDieSur20lavrentiev}}

We want to compare the approach of this paper with the methods of our previous paper \cite{BalDieSur20lavrentiev}. Though the results of the latter paper concern the Lavrentiev phenomenon and in this paper we are concerned with the irregularity of minimizers the approach is in fact very much alike.  Let us demonstrate this on the sub-dimensional case. 

Recall that in our proof we show the existence of ``upper'' and ``lower'' traces for a function from the domain of the integral functional. Above we implement this only for the minimizer but clearly the same construction can be applied to any function with certain integrability properties of the gradient. In the Section~\ref{ssec:trace} and below we worked with averages on balls but we can take in this construction any nonnegative $\omega\in C_0^\infty(B_{1/4}^d(0))$ with $\int \omega \, dx =1$. Then the upper and lower traces of a function $f$, which takes the boundary value $\eta u$ at the point $\bar x\in \frC^{d-1}_{\lambda,\gamma}$ are expressed as 
$$
f_{\pm}(\bar x) =\pm \frac{\eta}{2}+ \int\limits K^\pm (\bar x,z) \cdot\nabla f(z)\, dz
$$ 
with the kernel $K^\pm (\bar x,z) \lesssim |z_d|^{1-d} \indicator_{C^\pm(\bar x,0)}$ and is smooth outside the point $(\bar x,0)$. Integrating this with respect to the corresponding Cantor measure we get 
$$
\int (f_{+}(\bar x) - f_{-}(\bar x))\, d\mu(\bar x) = \eta + \int\limits_\Omega \mathbf{b}(z) \cdot\nabla f(z)\, dz
$$
where the kernel $\mathbf{b}$ satisfies $|\mathbf{b}|\lesssim b$ (here $b$ is our auxiliary function) and $\mathbf{b} \in C^\infty(\overline{\Omega} \setminus \frC^{d-1}_{\lambda,\gamma}\times\{0\})$. By construction $|\mathbf{b}|\cdot |\nabla u|\lesssim b \cdot |\nabla u|=0$ outside $\frC^{d-1}_{\lambda,\gamma}\times \{0\}$. 

Let $\xi\in C_0^\infty(\Omega)$ and consider the function $f = \xi + \eta u $. This function assumes the same boundary values as $\eta u$ and thus using that the upper and lower traces of the competitor $u$ on $\frC^{d-1}_{\lambda,\gamma}$ are $u_{\pm}=\pm 1/2$ we get  
$$
\eta=\eta + \int\limits_\Omega \mathbf{b}(z) \cdot \nabla \xi(z)\, dz.
$$
This means that the vector field $\mathbf{b}$ is solenoidal.  

Now let $\xi \in C^\infty_0(\Omega)$ be such that $\xi =1$ in a neighborhood of $\frC^{d-1}_{\lambda,\gamma}\times\{0\}$. Take $f= (1-\xi) \eta u$. Then $f\in C^\infty(\overline\Omega)$ and
$$
0 = \eta + \eta\int\limits_\Omega \mathbf{b} u\cdot \nabla (1-\xi)\, dz.
$$ 
Using that $u=\pm 1/2$ on $\mathrm{supp}\, \mathbf{b} \cap\{\pm x_d>0\}$ and $(1-\xi)\mathbf{b}=0$ on $(-1,1)^{d-1}\times\{0\}$, and integrating by parts we get
\begin{equation}\label{gauss}
\int \limits_{\partial \Omega} u\mathbf{b}\cdot \vec{\nu}\, d\sigma=1.
\end{equation}
Thus the vector field $\mathbf{b}$ has the same properties as the vector field used in  \cite{BalDieSur20lavrentiev} for constructing examples with the Lavrentiev phenomenon:  it is
\begin{itemize}
\item smooth in $\overline\Omega \setminus (\frC^{d-1}_{\lambda,\gamma}\times \{0\})$,

\item solenoidal in the sense of distributions in $\Omega$,

\item satisfies $|\mathbf{b}|\cdot |\nabla u|=0$ in $\Omega \setminus (\frC^{d-1}_{\lambda,\gamma}\times \{0\})$,

\item and satisfies \eqref{gauss}.
\end{itemize}
 
Indeed, for a vector field $\mathbf{b}$ with these properties and satisfying $\mathcal{F}^*(b)<\infty$ consider the functional 
$$
\mathcal{S}(f)= \int\limits_\Omega \mathbf{b}(z) \cdot\nabla f(z)\, dz.
$$
On smooth functions assuming the boundary value $u$ we have $\mathcal{S}(f)=1$ (by the Gauss theorem) while $\mathcal{S}(u)=0$, so in $W^{1,\phi(\cdot)}(\Omega)$ this functional separates the linear span of $u$ and the closure of smooth functions. See the remaining details in \cite{BalDieSur20lavrentiev}.


\section{Relation with the Lavrentiev gap}

Next we show that the approach used in this paper can be also applied to establishing the Lavrentiev gap. 

Let us state the corresponding conditional result in the spirit of Section~\ref{ssec:condition}.This result repeats one of the results in \cite{BalDieSur20lavrentiev}, but is obtained in a different way.

Let $\xi\in C_0^\infty(\Omega)$ be such that $\xi = 1$ in a neighborhood of $\frS = \frC^{d-1}_{\lambda,\gamma}\times \{0\}$ in the subdimensional case and $\frS = \{0\}^{d-1}\times \frC_{\lambda,\gamma}$ in the superdimensional case. Denote $u^\partial = (1-\xi) u$.

\begin{theorem}
Under Assumption~\ref{ass:harmonic2} (or Assumption~\ref{ass:harmonic3} with sufficiently small $\kappa>0$) for some $\eta>0$ there holds
\begin{equation}\label{LaGap}
\inf\{\mathcal{F}(v),\ v\in \eta u^\partial +C_0^\infty(\Omega)\} \geq 2\min\{\mathcal{F}(v),\ v\in \eta u +W_0^{1,\phi(\cdot)}(\Omega)\}.
\end{equation} 
\end{theorem}

\begin{proof}
First we recall the subdimensional case and the argument which led to showing the discontinuity of the minimizer (see Section~\ref{ssec:subdis}). Let $w_\eta\in \eta u^\partial+W_0^{1,\phi(\cdot)}(\Omega)$. Repeating the beginning of the proof of Theorem~\ref{T:diff_trace} we obtain subsequently
\begin{gather*}
|(w_\eta)_\pm (\bar x) \mp \eta/2 | \leq C(d) I_1^{\pm}[|\nabla w_\eta| \indicator_\Omega] (\bar x, 0),\\
\int\limits_{\frC} |(w_\eta)_\pm (\bar x) \mp \eta/2 |\, d\mu^{d-1}_{\lambda,\gamma}  (\bar x) \leq C(d) \int\limits_{(-1,1)^{d-1}} I_1^{\pm}[|\nabla w_\eta|\indicator_\Omega] (\bar x, 0)\, d \mu^{d-1}_{\lambda,\gamma}  (\bar x)\\
\leq C(d,\lambda,\gamma) \int\limits_\Omega  |\nabla w_\eta|b\, dz := J(\eta).
\end{gather*}
By the triangle inequality \eqref{triangle} we get
$$
\eta \leq \int\limits_{\frC} |(v_\eta)_+ (\bar x) - (v_\eta)_+ (\bar x)|\, d\mu^{d-1}_{\lambda,\gamma} (\bar x)+ 2 J(\eta) : = I(\eta)+II(\eta).
$$
This implies the following alternative: 
\begin{align*}
&\text{either}\quad I(\eta):=\int |(w_\eta)_{+}(\bar x) - (w_\eta)_{-}(\bar x)|\, d\mu(\bar x) \geq \frac{\eta}{2}\\
&\text{or}\quad II(\eta):=2J(\eta)\geq \frac{\eta}{2}.
\end{align*}
The first case $I(\eta)\geq \eta/2$ is the one we use in the proof of Theorem~\ref{T:diff_trace}, it gives the discontinuity of the $W$-minimizer, since the condition that $\mathcal{F}(w_\eta)\leq \mathcal{F}(\eta u)$ implies under Assumption~\ref{ass:harmonic2} (or Assumption~\ref{ass:harmonic3} with sufficiently small $\kappa$) that $II(\eta)<\eta/4$ for some $\eta>0$. 

On the other hand, for any function $w_\eta\in \eta u^\partial + C_0^\infty(\Omega)$ the difference of ``upper'' and ``lower'' traces is zero ($I(\eta)=0$), which leads us to the second case. By Lemmas~\ref{lem:riesz-sublinear00}, \ref{lem:riesz-sublinear01}, \ref{lem:riesz-sublinear} with $Q=2$, under Assumption~\ref{ass:harmonic2} (or Assumption~\ref{ass:harmonic3} with sufficiently small $\kappa$)
$$
\text{the inequality}\quad \mathcal{F}(w_\eta)\leq 2\mathcal{F}(\eta u)\quad \text{implies}\quad \int\limits_\Omega |\nabla w_\eta|b\, dz< \eta/4
$$
for some $\eta>0$. This means that 
$$
\mathcal{F}(w_\eta)\geq 2\mathcal{F}(\eta u).
$$ 
Recalling that $w_\eta$ is an arbitrary function assuming the boundary value $u\eta$ we see that the $H$-minimizer will have the energy greater than $2\mathcal{F}(\eta u)$. This is the Lavrentiev gap~\eqref{LaGap}.

Let us proceed to the superdimensional case. Recall the proof of Proposition~\ref{prop:sd1}, which was based on the estimate~\ref{superdM}. As above, we can apply the argument of Proposition~\ref{prop:sd1} to any function $w_\eta$ taking the same boundary values as $\eta u$, which gives (in the notation of Section~\ref{sec:super})
\begin{gather*}
I(\eta)+II(\eta):= \\
\sum_{j=1}^{2^m} (w_\eta(\bar{0},\xi^{+}_{m,j}) - w_\eta(\bar{0},\xi^{-}_{m,j}))+\sum_{j=0}^{2^m}(w_\eta(\bar{0}, \xi^{-}_{m,j+1}) - w_\eta(\bar{0}, \xi^{+}_{m,j}))  \\
=w_\eta(\bar 0,1)-w_\eta(\bar 0,-1)=(\eta/2)-(-\eta/2)=\eta.
\end{gather*}
Recall that here for $m\in \mathbb{N}$ and $j=1,\ldots,2^m$ by $\xi^{+}_{m,j}$ and $\xi^{-}_{m,j}$ we denote the right and left end points of pre-Cantor intervals of the $m$-th generation, with $\xi^+_{m,0}=-1$ and $\xi^-_{m,2^m}=1$. Thus 
$$
\text{either}\quad I(\eta)\geq \eta/2\quad \text{or}\quad II(\eta)\geq \eta/2.
$$
The case $I(\eta)\geq \eta/2$ is the one we use in the proof above (for $w_\eta=v_\eta$), since by Assumption~\ref{ass:harmonic2} (or Assumption~\ref{ass:harmonic3} with small $\kappa>0$) for the minimizer $v_\eta$ with some $\eta>0$ we get $II(\eta)\leq \eta/4$. 

On the other hand, for any smooth function $I(\eta)\to 0$ as $m\to 0$ (by the absolute continuity), so for large $m$ we fall into the second case $II(\eta)\geq \eta/2$. By the estimate \eqref{superdM} this implies 
$$
C(d)\int\limits_\Omega |\nabla w_\eta|b\,dx \geq  \eta/2.
$$
By Lemmas~\ref{lem:riesz-sublinear00}, \ref{lem:riesz-sublinear01} with $Q=2$, under Assumption~\ref{ass:harmonic2} (or Assumption~\ref{ass:harmonic3} with sufficiently small $\kappa$)
$$
\text{the inequality}\quad\mathcal{F}(w_\eta)\leq 2\mathcal{F}(\eta u)\quad \text{implies}\quad C(d)\int\limits_\Omega |\nabla w_\eta| b\, dz< \eta/4
$$
for some $\eta>0$. This means that 
$$
\mathcal{F}(w_\eta)\geq 2\mathcal{F}(\eta u)
$$
and thus the $H$-minimizer will have the energy greater than $2\mathcal{F}(\eta u)$. This is the Lavrentiev gap~\eqref{LaGap}.

\end{proof}




The  research of Anna Balci and Lars Diening is funded  by the Deutsche Forschungsgemeinschaft (DFG, German Research Foundation) - SFB 1283/2 2021 - 317210226 at the Bielefeld University. The Research of Anna  Balci is  supported by Charles University  PRIMUS/24/SCI/020 and Research Centre program \\ No. UNCE/24/SCI/005.  

\printbibliography

\end{document}